\documentclass[12pt]{amsart}

\usepackage{amssymb}

\makeatletter
\let\th@plain=\undefined
\makeatletter
\let\proofname=\undefined

\usepackage[pdftex]{hyperref} 
\usepackage[amsmath,amsthm,thmmarks,hyperref]{ntheorem}
\usepackage[capitalize]{cleveref}

\newcommand{\pref}[1]{(\ref{#1})}
\newcommand{\fullcref}[2]{\cref{#1}\pref{#1-#2}}
\newcommand{\csee}[1]{(see \cref{#1})}

\newcommand{\iso}{\cong}
\newcommand{\normal}{\triangleleft}
\newcommand{\notnormal}{\mathrel{\not\hskip-2pt\normal}}

\newcommand{\integer}{\mathbb{Z}}

\DeclareMathOperator{\Cay}{Cay}
\DeclareMathOperator{\GL}{GL}

\makeatletter
\renewenvironment{proof}[1][\proofname]{
\vskip-\lastskip \par \goodbreak \vskip 6pt plus6pt  %% I did this 4/26/08
  \th@nonumberplain
  \normalfont
  \theoremsymbol{\ensuremath{_\blacksquare}}
  \@thm{proof}{proof}{#1}}%
  {\@endtheorem}
\makeatother

\makeatletter
\renewtheoremstyle{plain}%
{\medbreak\item[\hskip\labelsep \normalfont(##2) \theorem@headerfont ##1\@addpunct{.}]}%
{\medbreak\item[\hskip\labelsep \normalfont(##2) \theorem@headerfont ##1 \normalfont ##3\@addpunct{.}]}

\renewtheoremstyle{definition}%
{\medbreak\item[\hskip\labelsep \normalfont(##2) \theorem@headerfont ##1\@addpunct{.}]\normalfont}%
{\medbreak\item[\hskip\labelsep \normalfont(##2) \theorem@headerfont ##1 \normalfont ##3\@addpunct{.}] \normalfont}

\renewtheoremstyle{remark}%
{\medbreak\item[\hskip\labelsep \normalfont(##2) \normalfont\itshape ##1\@addpunct{.}]\normalfont}%
{\medbreak\item[\hskip\labelsep \normalfont(##2) \normalfont\itshape ##1 \normalfont ##3\@addpunct{.}] \normalfont}

\renewtheoremstyle{nonumberremark}%
{\medbreak\item[\hskip\labelsep  \normalfont\itshape ##1\@addpunct{.}]\normalfont}%
{\medbreak\item[\hskip\labelsep \normalfont\itshape ##1 \normalfont ##3\@addpunct{.}] \normalfont}

\makeatother

 \newenvironment{claim}[1][\unskip]{\bf
 \medskip \noindent Claim. \it}{\unskip\upshape}
\crefformat{case}{Case~#2#1#3}

 \newcounter{case}
 \newenvironment{case}[1][\unskip]{\refstepcounter{case}\bf
 \medskip \noindent Case \thecase\ #1. \it}{\unskip\upshape}
 \renewcommand{\thecase}{\arabic{case}}
\crefformat{case}{Case~#2#1#3}

 \newcounter{subcase}
 \newenvironment{subcase}[1][\unskip]{\refstepcounter{subcase}\bf
 \medskip \noindent \hskip\parindent Subcase \thesubcase\ #1. \it}{\unskip\upshape}
\numberwithin{subcase}{case}
\crefformat{subcase}{Subcase~#2#1#3}

 \newcounter{subsubcase}
 \newenvironment{subsubcase}[1][\unskip]{\refstepcounter{subsubcase}\bf
 \medskip \noindent \hskip2\parindent Subsubcase \thesubsubcase\ #1. \it}{\unskip\upshape}
\numberwithin{subsubcase}{subcase}
\crefformat{subsubcase}{Subsubcase~#2#1#3}

\crefname{subsubcase}{Subsubcase}{Subsubcases}

 \newcounter{subsubsubcase}
 \newenvironment{subsubsubcase}[1][\unskip]{\refstepcounter{subsubsubcase}\bf
 \medskip \noindent \hskip3\parindent Subsubsubcase \thesubsubsubcase\ #1. \it}{\unskip\upshape}
\numberwithin{subsubsubcase}{subsubcase}
\crefformat{subsubsubcase}{Subsubsubcase~#2#1#3}

\numberwithin{equation}{section}

\theoremstyle{plain}
\newtheorem{thm}[equation]{Theorem}
\newtheorem{prop}[equation]{Proposition}
\newtheorem{cor}[equation]{Corollary}

\newtheorem{lem}[equation]{Lemma}
\newtheorem{FGL}[equation]{Lemma}

\crefformat{thm}{Theorem~#2#1#3}
\crefformat{prop}{Proposition~#2#1#3}
\crefformat{goal}{Goal~#2#1#3}
\crefformat{lem}{Lemma~#2#1#3}
\crefformat{FGL}{the Factor Group Lemma~#2(#1)#3}

\crefname{lem}{Lemma}{Lemmas}
\crefname{cor}{Corollary}{Corollaries}

\crefmultiformat{prop}{Propositions~#2#1#3}{ and~#2#1#3}{, #2#1#3}{, and~#2#1#3}
\crefmultiformat{lem}{Lemmas~#2#1#3}{ and~#2#1#3}{, #2#1#3}{, and~#2#1#3}

\theoremstyle{remark}
\newtheorem{rem}[equation]{Remark}
\newtheorem{obs}[equation]{Observation}
\newtheorem*{ack}{Acknowledgments}[section]
\newtheorem*{notation}{Notation}[section]

\newcounter{saveBibCtr}

\begin{document}

\title[Cayley graphs of order $16p$ are hamiltonian]
{Cayley graphs of order $\boldsymbol{16p}$ are hamiltonian}

\author[S.\,J.\,Curran]{Stephen J. Curran}
\address{Department of Mathematics, 
University of Pittsburgh at Johnstown, 
Johnstown, PA 15904, USA}
\email{SJCurran@pitt.edu}

\author[D.\,W.\,Morris]{Dave Witte Morris}
 \address{Department of Mathematics and Computer Science,
 University of Lethbridge,
 Lethbridge, Alberta, T1K~3M4, Canada}
 \email{Dave.Morris@uleth.ca, 
 \href{http://people.uleth.ca/~dave.morris/}{http://people.uleth.ca/$\sim$dave.morris/}
 }

\author[J.\,Morris]{Joy Morris}
 \address{Department of Mathematics and Computer Science,
 University of Lethbridge,
 Lethbridge, Alberta, T1K~3M4, Canada}
 \email{Joy.Morris@uleth.ca, 
 \href{http://people.uleth.ca/~morris/}{http://people.uleth.ca/$\sim$morris/}
 }

\date{\today}

\begin{abstract}
Suppose $G$ is a finite group, such that $|G| = 16p$, where $p$~is prime. We show that if $S$ is any generating set of~$G$, then there is a hamiltonian cycle in the corresponding Cayley graph $\Cay(G;S)$.
\end{abstract}

\maketitle

\section{Introduction} \label{intro}

This paper establishes one of the cases of Theorem~1.2(1) % is the reference still correct???
of \cite{M2Slovenian-LowOrder}. Namely, several of the main results of that paper combine to show:

	\begin{align} \label{kp}
	\begin{matrix}
	\text{\it Every connected Cayley graph on~$G$ has a hamiltonian cycle}
	\\ \text{\it if\/ $|G| = kp$, 
	where $p$~is prime, $1 \le k < 32$, and $k \notin \{16, 24, 27, 30\}$.}
	\end{matrix}
	\end{align}
We handle the first excluded case:

\begin{thm} \label{16p}
If $|G| = 16p$, where $p$~is prime, then every connected Cayley graph on~$G$ has a hamiltonian cycle.
\end{thm}

\begin{rem}
The cases $k = 27$ and $k = 30$ are covered in \cite{GhaderpourMorris-27p,GhaderpourMorris-30p}, but it seems that the case $k = 24$ will be more difficult.
\end{rem}

Here is an outline of the paper:
	\begin{enumerate}
	\item[\ref{intro}.] Introduction
	\item[\ref{prelim-hamcyc}.] Preliminaries on hamiltonian cycles in Cayley graphs
		\begin{enumerate}
		\item[\ref{FGLSect}.] Factor Group Lemma
		\item[\ref{CyclicNormalSect}.] Generator in a cyclic, normal subgroup
		\item[\ref{MiscSect}.]  Miscellaneous results
		\end{enumerate}
	\item[\ref{pNotNormalSect}.] Groups without a normal Sylow $p$-subgroup
		\begin{enumerate}
		\item[\ref{pNotNormalSect-48}.] Groups of order $48$
		\item[\ref{pNotNormalSect-80}.] Groups of order $80$
		\item[\ref{pNotNormalSect-112}.] Groups of order $112$
		\end{enumerate}
	\item[\ref{prelim-16p}.] Preliminaries on groups of order~$16p$
	\item[\ref{Z2xZ2Sect}.] The case where $G/G' \iso \integer_2 \times \integer_2$
	\item[\ref{Z4xZ2Sect}.] The case where $G/G' \iso \integer_4 \times \integer_2$
	\item[\ref{Z2xZ2xZ2Sect}.] The case where $G/G' \iso \integer_2 \times \integer_2 \times \integer_2$
	\end{enumerate}

\begin{ack}
This research was partially supported by a grant from the Natural Sciences and Engineering Research Council of Canada.
Most of the work was carried out during a visit of S.J.C.\ to the University of Lethbridge. 
\end{ack}

\section{Preliminaries on hamiltonian cycles in Cayley graphs} \label{prelim-hamcyc}

For ease of reference, we reproduce several useful results that provide hamiltonian cycles in Cayley graphs.

\begin{notation}
For any group~$G$, and any $a,b \in G$, we use:
	\begin{enumerate}
	\item $[a,b]$ to denote the \emph{commutator} $a^{-1} b^{-1} a b$,
	\item $b^a$ to denote the \emph{conjugate} $a^{-1} b a$,
	\item $G'$ to denote the \emph{commutator subgroup} $[G,G]$ of~$G$,
	and
	\item $\Phi(G)$ to denote the \emph{Frattini subgroup} of~$G$.
	\end{enumerate}
See \cite[\S5.1]{Gorenstein-FinGrps} for some basic properties of the Frattini subgroup (and its definition).
\end{notation}

\subsection{Factor Group Lemma}
\label{FGLSect}

The following elementary results are well known (and easy to prove).

\begin{FGL}[(``Factor Group Lemma'' {\cite[\S2.2]{WitteGallian-survey}})] \label{FGL}
Suppose
 \begin{itemize}
 \item $N$ is a cyclic, normal subgroup of~$G$,
 \item $(s_1,s_2,\ldots,s_m)$ is a hamiltonian cycle in $\Cay(G/N;S)$,
 and
 \item the product $s_1s_2\cdots s_m$ generates~$N$.
 \end{itemize}
 Then $(s_1,s_2,\ldots,s_m)^{|N|}$ is a hamiltonian cycle in $\Cay(G;S)$.
 \end{FGL}

\begin{cor} \label{FGL(gen-twice)}
  Suppose
 \begin{itemize}
 \item $N$ is a cyclic, normal subgroup of~$G$, such that $|N|$ is a prime power,
  \item $\langle s^{-1} t \rangle = N$ for some $s,t \in S \cup S^{-1}$,
  and
 \item there is a hamiltonian cycle in $\Cay(G/N;S)$ that uses at least one edge labelled~$s$.
 \end{itemize}
 Then there is a hamiltonian cycle in $\Cay(G;S)$.
 \end{cor}

\begin{cor} \label{FGL(order2)}
  Suppose
 \begin{itemize}
 \item $N$ is a cyclic, normal subgroup of~$G$, such that $|N|$ is a prime power,
  \item $s \in S$, with $s^2 \in N \smallsetminus \Phi(N)$,
    and
 \item there is a hamiltonian cycle in $\Cay(G/N;S)$ that uses at least one edge labelled~$s$.
 \end{itemize}
 Then there is a hamiltonian cycle in $\Cay(G;S)$.
 \end{cor}

\begin{lem}[{}{(\cite[Cor.~2.9]{M2Slovenian-LowOrder}}] \label{MultiDouble}
  Suppose
 \begin{itemize}
 \item $S$ is a generating set of~$G$,
 \item $H$ is a subgroup of~$G$, such that $|H|$ is prime,
 \item the quotient multigraph $H \backslash \Cay(G;S)$ has a hamiltonian cycle~$C$,
 and
 \item $C$ uses some double-edge of~$H \backslash \Cay(G;S)$.
 \end{itemize}
 Then there is a hamiltonian cycle in $\Cay(G;S)$.
 \end{lem}

\subsection{Generator in a cyclic, normal subgroup}
\label{CyclicNormalSect}

\begin{thm}[{}{(Alspach \cite[Cor.~5.2]{Alspach-lifting})}]
 \label{AlspachSemiProd}
 Suppose
 \begin{itemize}
 \item $s$ and~$t$ are elements of~$G$,
 and
 \item $G = \langle s \rangle \ltimes \langle t \rangle$.
 \end{itemize}
 Then $\Cay(G;s,t)$ has a hamiltonian cycle.
 \end{thm}

The following observation is well known:

\begin{lem}[{}{\cite[Lem.~2.27]{M2Slovenian-LowOrder}}]
\label{NormalEasy}
 Let $S$ generate $G$ and let $s \in S$, such that $\langle s \rangle
\normal G$. If
 \begin{itemize}
 \item $\Cay \bigl( G/\langle s \rangle ; S \bigr)$ has a
hamiltonian cycle,
 and
 \item either
 \begin{enumerate}
 \item \label{NormalEasy-Z} 
 $s \in Z(G)$,
% or
% \item \label{NormalEasy-notZ} 
% $Z(G) \cap \langle s \rangle = \{e\}$,
 or
 \item \label{NormalEasy-p}
 $|s|$ is prime,
 \end{enumerate}
 \end{itemize}
 then $\Cay(G;S)$ has a hamiltonian cycle.
 \end{lem}

Here is another similar result:

\begin{lem} \label{CyclicNormal2p}
  Suppose
 \begin{itemize}
 \item $s \in S$, with $\langle s \rangle \normal G$,
 \item $|s|$ is a divisor of $pq$, where $p$ and~$q$ are distinct primes,
 \item $s^p \in Z(G)$,
 \item $|G/\langle s \rangle|$ is divisible by~$q$,
 and
 \item $\Cay \bigl( G/\langle s \rangle ; S \bigr)$ has a hamiltonian cycle.
 \end{itemize}
 Then there is a hamiltonian cycle in $\Cay(G;S)$.
 \end{lem}

\begin{proof}
We may assume $|s| = pq$ and $s \notin Z(G)$, for otherwise \cref{NormalEasy} applies. 
Let 
	\begin{itemize}
	\item $(s_1,\ldots,s_m)$ be a hamiltonian cycle in $\Cay \bigl( G/\langle s \rangle ; S \bigr)$, 	\item $g = s_1s_2\cdots s_m$ be its endpoint in~$G$,
	and
	\item $k = | s | / | g |$.
	\end{itemize}
Consider the walk
	$$ (s_1, s^{k-1},s_2, s^{k-1},\ldots,s_m, s^{k-1}) .$$
Writing $s = xw$, where $x$~is the $q$-part of~$s$ and $w$~is the $p$-part of~$s$, and noting that $x \in Z(G)$ (because $x^p = s^p \in Z(G)$), we see that the endpoint is
	\begin{align} \label{CyclicNormal2p-endpt}
	s_1 (xw)^{k-1} s_2  (xw)^{k-1} \cdots s_m  (xw)^{k-1}
	= g \, x^{(k-1)m} \prod_{g' \in G/\langle s \rangle} w^{g'}
	= g,
	\end{align}
since $m = |G/\langle s \rangle|$ is divisible by~$q$, and $\langle w \rangle \cap Z(G) = \{e\}$.

Therefore, the walk
	$$ (s_1, s^{k-1},s_2, s^{k-1},\ldots,s_m, s^{k-1})^{|g|} $$
is closed. Also (using \pref{CyclicNormal2p-endpt}), it is not difficult to see that the walk traverses all of the elements of~$G$. Therefore, it is a hamiltonian cycle in $\Cay(G;S)$.
\end{proof}

\subsection{Miscellaneous results}
\label{MiscSect}

\begin{thm}[(Maru\v si\v c-Durnberger-Keating-Witte \cite{KeatingWitte})] \label{KeatingWitte}
 If $G'$ is a cyclic $p$-group, then
every connected Cayley graph on~$G$ has a hamiltonian cycle.
 \end{thm}

The proof in \cite{Witte-pn} yields the following result:

\begin{thm}[{}{\cite[Cor.~3.3]{Morris-2genNilp}}] \label{pkSubgrp}
  Suppose
 \begin{itemize}
 \item $S$ is a generating set of~$G$,
 \item $N$ is a normal $p$-subgroup of~$G$,
 and
 \item $s t^{-1} \in N$, for all $s,t \in S$.
 \end{itemize}
 Then $\Cay(G;S)$ has a hamiltonian cycle.
 \end{thm}

The following observation is also known, but we do not know whether it is in the literature, so we provide a proof. Because it is of independent interest, we prove a more general version than we need.

\begin{lem} \label{CxLHamCyc}
Let $S$ generate~$G$, and let $X$ be a subset of~$S$. Assume:
	\begin{itemize}
	\item $\langle X \rangle$ is abelian {\upshape(}and nontrivial\/{\upshape)},
	\item for each $g \in G$, either $g$ centralizes every element of~$\langle X \rangle$, or $g$~inverts every element of~$\langle X \rangle$,
	\item there is a hamiltonian cycle in $\Cay \bigl( \langle X \rangle; X \bigr)$,
	and
	\item there is a hamiltonian path in $\Cay(G/\langle X \rangle; S)$.
	\end{itemize}
Then there is a hamiltonian cycle in $\Cay(G;S)$.
\end{lem}

\begin{proof}
Let
	\begin{itemize}
	\item $[x_0,x_1,\ldots,x_m]$ be a hamiltonian cycle in $\Cay\bigl( \langle X \rangle; X \bigr)$,
	\item $[g_0,g_1,\ldots,g_n]$ be a path in $\Cay(G;S)$ that is the lift of a hamiltonian path in $\Cay \bigl( G/\langle X \rangle; S \bigr)$,
%	\item $\epsilon_i \in \{\pm1\}$, such that $x^{g_i} = x^{\epsilon_i}$ for all $x \in X$,
	\item $C = \Cay \bigl( \integer_m; \{1\} \bigr)$ be a cycle of length~$m$,
	\item $L$ be the path of length~$n$ with consecutive vertices $0,1,\ldots,n$,
	\item $f \colon V(C) \times V(L) \to G$ be defined by
		$$ f(i,j) = x_i \, g_j .$$
	\end{itemize}
Note that:
	\begin{itemize}
	\item for $0 \le i < m$ and $0\le j < n$,  we have
		$$f(i,j)^{-1} \cdot f(i,j+1) 
		= \bigl( x_i \, g_j  \bigr)^{-1} \bigl( x_i \, g_{j+1}\bigr)
		= g_j^{-1} g_{j+1}
		\in S \cup S^{-1}
		,$$
	because $g_j$ and $g_{j+1}$ are adjacent vertices in $\Cay\bigl( G; S \bigr)$,
	and
	\item for $0 \le i < m$ and $0\le j \le n$,
	and letting $x = x_i^{-1} x_{i+1} \in X \cup X^{-1}$, we have
		$$f(i,j)^{-1} \cdot f(i+1,j) 
		= \bigl( x_i \, g_j \bigr)^{-1} \bigl( x_{i+1} \, g_j \bigr)
		= g_j^{-1} x g_j
		= x^{\pm 1}
		\in X \cup X^{-1}
		.$$
	\end{itemize}
Thus, $f$ is an isomorphism from the Cartesian product $C \times L$ onto a subgraph of $\Cay(G;S)$. Since the two graphs have the same number of vertices, it is a spanning subgraph. Then, since it is easy to see that $C \times L$ has a hamiltonian cycle \cite[Corollary on p.~29]{ChenQuimpo-hamconn}, we conclude that $\Cay(G;S)$ has a hamiltonian cycle.
\end{proof}

\begin{rem} \label{QuotientIs4pOr8p}
When we apply \cref{FGL(gen-twice)} or \cref{FGL(order2)} 
to obtain a hamiltonian cycle in $\Cay(G;S)$, and $G$ has order $16p$,
the order of $G/N$ is either $4p$ or $8p$. Thus, \pref{kp} provides a hamiltonian cycle in $\Cay(G/N;S)$ with at least one edge labelled~$s$.
Similarly,  \pref{kp} provides a hamiltonian cycle in $\Cay(G/\langle s\rangle;S)$ for \cref{NormalEasy,CyclicNormal2p}, and it provides a hamiltonian path in $\Cay(G/\langle X\rangle;S)$ for \cref{CxLHamCyc}.
\end{rem}

\section{Groups without a normal Sylow $p$-subgroup} \label{pNotNormalSect}

In this section, we prove \cref{16p} under the additional assumption that the Sylow $p$-subgroups of~$G$ are not normal:

\begin{prop} \label{pNotNormal}
If\/ $|G| = 16p$, where $p$ is prime, and the Sylow $p$-subgroups of~$G$ are not normal, then every connected Cayley graph on~$G$ has a hamiltonian cycle.
\end{prop}

We first note that there are only three possibilities for the order of~$G$:

\begin{lem} \label{NotNormal->48or80or112}
If\/ $|G| = 16p$, where $p$~is prime, and the Sylow $p$-subgroups of~$G$ are not normal, then $p \in \{3,5,7\}$, so $|G| \in \{48, 80, 112\}$.
\end{lem}

\begin{proof}
By Sylow's Theorem \cite[Thm.~15.7, p.~230]{Judson-AlgText}, we know that the number of Sylow $p$-subgroups is a divisor of~$16$, and is congruent to~$1$, modulo~$p$. Since the only prime divisors of $2 - 1 = 1$, $4 - 1 = 3$, $8 - 1 = 7$, and $16 - 1 = 15 = 3 \times 5$ are $3$, $5$, and~$7$, this implies there is only one Sylow $p$-subgroup (which is normal) unless $p \in \{3,5,7\}$.
\end{proof}

We now list the nine nonabelian groups of order $16$, The list will be used repeatedly in the remainder of the paper, because each of these nine groups arises as the Sylow $2$-subgroup of a group of order~$16p$.

\begin{prop}[{}{\cite[\S118, p.~146]{Burnside-ThyGrps}}] \label{ListOrder16}
%[({\cite[Table~1, p.~134]{CoxeterMoser-GensAndRels}})] \label{ListOrder16}
There are $9$ nonabelian groups of order~$16$:
	\begin{enumerate} \itemsep=\medskipamount
	
	\item \label{ListOrder16-Z2xZ2}
$3$ groups with $Q/Q' \iso \integer_2 \times \integer_2$: 
		\begin{enumerate}\itemsep=\smallskipamount \smallskip
		\item $D_{16}$ {\upshape(}\!\!``dihedral''{\upshape)}, 
		\item $Q_{16}$ {\upshape(}\!\!``generalized quaternion''{\upshape)}, 
		and
		\item $\integer_2 \ltimes \integer_8 = \langle x \rangle \ltimes \langle y \rangle$ with $x^{-1} y x = y^3$ {\upshape(}\!\!``semidihedral'' or ``quasidihedral''{\upshape)}.
		\end{enumerate}
		
	\item \label{ListOrder16-Z4xZ2}
	$3$ groups with $Q/Q' \iso \integer_4 \times \integer_2$:
		\begin{enumerate} \itemsep=\smallskipamount \smallskip
		\item  \label{ListOrder16-Z4xZ2-Z2xZ8}
$\integer_2 \ltimes \integer_8 = \langle x \rangle \ltimes \langle y \rangle$ with $x^{-1} y x = y^5$,
		\item  \label{ListOrder16-Z4xZ2-Z4xZ4}
$\integer_4 \ltimes \integer_4 = \langle x \rangle \ltimes \langle y \rangle$ with $x^{-1} y x = y^{-1}$, 
		and
		\item  \label{ListOrder16-Z4xZ2-Z4x(Z2xZ2)}
$\integer_4 \ltimes (\integer_2 \times \integer_2) = \langle x \rangle \ltimes \langle y,z \rangle$ with $x^{-1} y x = y z$ and $x^{-1} z x = z$.
		\end{enumerate}
		
	\item \label{ListOrder16-Z2xZ2xZ2}
	$3$ groups with $Q/Q' \iso \integer_2 \times \integer_2 \times \integer_2$:
		\begin{enumerate} \itemsep=\smallskipamount \smallskip
		\item \label{ListOrder16-Z2xZ2xZ2-D8xZ2}
$D_8 \times \integer_2 = \langle f, t \mid f^2 = t^4 = (ft)^2 = e \rangle \times \langle z \rangle$,
		\item \label{ListOrder16-Z2xZ2xZ2-Q8xZ2}
$Q_8 \times \integer_2$,
		and
		\item \label{ListOrder16-Z2xZ2xZ2-Z2x(Z2xZ4)}
$\integer_2 \ltimes (\integer_2 \times \integer_4) = \langle x \rangle \ltimes \langle y, z \rangle$ with $x^{-1} y x = yz^2$ and $x^{-1} zx = z$.
		\end{enumerate}
	\end{enumerate}
\end{prop}

The three possible orders of~$G$ are discussed individually, in \cref{48NotNormal,80NotNormal,112NotNormal} below.

\subsection{Groups of order~$48$} \label{pNotNormalSect-48}

\begin{prop}[{}{\cite{CurranMorris-48}}] \label{48NotNormal}
If\/ $|G| = 48$, and the Sylow $3$-subgroups of~$G$ are not normal, then every connected Cayley graph on~$G$ has a hamiltonian cycle.
\end{prop}

\begin{proof}[Comments on the proof]
A computer search can find hamiltonian cycles in all of these Cayley graphs fairly quickly.
Alternatively, a proof can be written by hand, but, unfortunately, our presentation of this \cite{CurranMorris-48} is an unilluminating, $15$-page case-by-case analysis, so we omit the details. 
\end{proof}

It would be interesting to have a conceptual proof of \cref{48NotNormal}, or, failing that, a human-readable proof of only 2 or 3 pages.

\subsection{Groups of order $80$} \label{pNotNormalSect-80}

\begin{prop} \label{80NotNormal}
If\/ $|G| = 80$, and the Sylow\/ $5$-subgroups of~$G$ are not normal, then every connected Cayley graph on~$G$ has a hamiltonian cycle.
\end{prop}

\begin{proof}
From Sylow's Theorem (and the observation that $16$ is the only nontrivial divisor of~$16$ that is congruent to~$1$ modulo~$5$), we know there are $16$ Sylow $5$-subgroups. These contain $16 \times 4 = 64 = |G| - 16$ nonidentity elements of~$G$, so the Sylow $2$-subgroup must be normal. Therefore $G = \integer_5 \ltimes Q$, where $Q$ is the Sylow $2$-subgroup. Since $\integer_5 \notnormal G$, we know the action on~$Q$ is nontrivial. 

We claim $G$ is isomorphic to a semidirect product $\integer_5 \ltimes (\integer_2)^4$. If not, then $Q$ is not elementary abelian, so $Q/\Phi(Q)$
has order $2$, $4$, or~$8$. Since groups of order $2$, $4$, or~$8$ have no automorphisms of order~$5$, this implies that $\integer_5$ acts trivially on $Q/\Phi(Q)$. Therefore $\integer_5$ acts trivially on~$Q$ \cite[Thm.~5.3.5]{Gorenstein-FinGrps}. This is a contradiction.

Now let $S$ be a minimal generating set for~$G$.
Then $S$ must contain an element~$x$ that generates $G/(\integer_2)^4$. Then $|x| = 5$, so, by passing to a conjugate, we may assume $\langle x \rangle = P$. Also, since $|x| = 5$, we know that $x$~acts on $(\integer_2)^4$ via a linear transformation whose minimal polynomial is $\lambda^4 + \lambda ^3 + \lambda ^2 + \lambda + 1$. Therefore, with respect to any basis of the form $\{v, v^x, v^{x^2}, v^{x^3}\}$, 
	$$ \text{$x$~acts via multiplication on the right by the matrix 
	$ A = \text{\smaller[2]$\begin{bmatrix}
	0 & 1 & 0 & 0 \\
	0 & 0 & 1 & 0 \\
	0 & 0 & 0 & 1 \\
	1 & 1 & 1 & 1 
	\end{bmatrix}$}
	$}
	. $$
(This is ``Rational Canonical Form.'')

Let $s$ be another element of~$S$. Then $\langle x, s \rangle$ has nontrivial intersection with $(\integer_2)^4$. Since $\GL_k(2)$ does not have any elements of order~$5$ when $k < 4$, we know that $x$ acts irreducibly, so this implies that $\langle x, s \rangle$ contains all of~$(\integer_2)^4$. Therefore $S = \{x, s\}$ (if $S$ is minimal). Obviously, $s$~is of the form $s = x^i v$, for some $v \in (\integer_2)^4$ and (by passing to the inverse if necessary) we may assume $0 \le i \le 2$.  If $i = 1$, then $x^{-1} s \in (\integer_2)^4$, so \cref{pkSubgrp} applies. Thus, we may assume $i \in \{0,2\}$. So 
	$$ \text{$S = \{x, v\}$ or $S = \{x, x^2v\}$} ,$$
and (by choosing an appropriate basis) $v = (1,0,0,0) \in (\integer_2)^4$.

\setcounter{case}{0}

\begin{case}
Assume $S=\{x,v\}$. 
\end{case}
We claim that a hamiltonian cycle in $\Cay(G;S)$ is given by:
\begin{align*}
&x^4, v, x^{-2}, v, x, v, (x^2,v)^3, (x,v)^2,
x^{-2}, v, x, v, x^{-1}, v, (x^2,v)^2, x^{-2}, v, x, v, x^2, v,\\
&x^{-2}, v, (x^{-1},v)^2, x, v, x^2, v,
x^{-2}, v, x^{-1}, v, x, v, (x^2,v)^2, (x^{-1},v)^3, x^2, v, x, v.
\end{align*}
To verify this, we list the vertices of the cycle, using $a$, $b$, $c$, and $d$ to
denote the generators of $(\integer_2)^4$, where $a=v$, $b=v^{x}$, $c=b^{x}=v^{x^2}$, and $d=c^{x}=v^{x^3}$.
Then $d^x=v^{x^4}=abcd$. 
The hamiltonian cycle visits the vertices of $\Cay(G;S)$ in the order:
\begin{align*}
&e, x, x^2, x^3, x^4, bx^4, bx^3, bx^2, bdx^2, bdx^3,
bcdx^3, bcdx^4, bcd, abcd, abcdx, abcdx^2, abcx^2,  \\
&abcx^3, abcx^4, acx^4, ac, c, cx, abdx, abd, abdx^4, 
adx^4, ad, d, dx^4, bdx^4, bd, bdx, acx, acx^2, \\
&acx^3, ax^3, ax^2, ax, bcdx, bcdx^2, bcx^2, bcx^3, 
bcx^4, cx^4, cx^3, cx^2, cdx^2, cdx, abx, ab, b, bx, \\
&acdx, acdx^2, acdx^3, adx^3, adx^2, adx, bcx, bc, abc, 
abcx, dx, dx^2, dx^3, cdx^3, cdx^4, cd,  \\
&acd, acdx^4, abcdx^4, abcdx^3, abdx^3, abdx^2, 
abx^2, abx^3, abx^4, ax^4, a, e.
\end{align*}

\begin{case}
Assume $S=\{x,x^2v\}$.
\end{case}
A hamiltonian cycle in the quotient multigraph $P \backslash {\Cay(G;S)}$ is given by:
\begin{align*}
x^2v, x^{-1}, (x^2v)^{-1}, x^{-4}, x^2v, x, (x^2v)^{-1}, x^{-1}, (x^2v)^{-1}, x^2, (x^2v)^2.
\end{align*}
Again we use the notation $a=v$, $b=v^{x}$, $c=v^{x^2}$, and $d=v^{x^3}$ to list the vertices in this hamiltonian cycle:
\begin{align*}
P, Pa, Pabcd, Pcd, Pbc, Pab, Pbcd, Pabc, Pb, Pc, Pad, Pabd, Pacd, Pac, Pbd, Pd, P
. \end{align*}
The edge from $Pac$ to $Pbd$ is a double edge, coming from both $x$ and $x^2v$, so \cref{MultiDouble} provides a hamiltonian cycle in $\Cay(G;S)$.
 \end{proof}

\subsection{Groups of order~$112$} \label{pNotNormalSect-112}

Before finding a hamiltonian cycle in $\Cay(G;S)$, we prove two results that determine the structure of~$G$.

\begin{lem} \label{112->NormalSylow}
If $G$ is any group of order~$112$, then $G$ has a normal Sylow subgroup.
\end{lem}

\begin{proof}[Proof by contradiction]
Assume $G$ has no normal Sylow subgroups, and let $P$ be a Sylow $7$-subgroup of~$G$.
Let $N$ be a minimal normal subgroup of~$G$. Since $G$ is solvable (for example, this follows from Burnside's $p^a q^b$ Theorem \cite[Thm.~4.3.3]{Gorenstein-FinGrps}),
$N$~is an elementary abelian normal subgroup of~$G$. Since $P$ is not normal, we must have $|N| = 2^k$ for some~$k$. 

\setcounter{case}{0}

\begin{case}
Assume $k \neq 3$.
\end{case}
We know $k \neq 4$, since the Sylow $2$-subgroups are not normal, so $k \in \{1,2\}$. Furthermore, we know that the Sylow $2$-subgroups of $G/N$ are not normal. 
Observe that $|G/N|$ is either $28$ or $56$. 

We claim that $PN\normal G$. If not, then the Sylow $7$-subgroup of $G/N$ is not normal, so $|G/N|=56$ and $G/N$ has eight Sylow $7$-subgroups. Thus, there are $|G/N|-|QN/N|=56-8=48$ elements of order 7 in $G/N$. So $G/N$ has only one Sylow $2$-subgroup, which must be normal. This contradicts the assumption that $G$ has no normal Sylow subgroups.

Since $|PN| = 7 |N| \in \{14, 28\}$, we know $P$ is normal (hence characteristic) in $PN$, so $P \normal G$. This contradicts the assumption that $G$ has no normal Sylow subgroups.

\begin{case}
Assume $k = 3$.
\end{case}
Then $|G/(PN)| = 2$, so $PN \normal G$. Since $P$ does not centralize~$N$ (otherwise $P$ would be normal in~$G$), it must act on~$N$ via a linear transformation of order~$7$. 

Since there is no normal $7$-complement, we know there is an element of~$G$ that normalizes~$P$, but does not centralize it \cite[Thm.~7.4.3]{Gorenstein-FinGrps}. So some element of~$G$ inverts~$P$, which means that every element of~$P$ is conjugate to its inverse. 

However, if we let 
	\begin{itemize}
	\item $g$ be a generator of~$P$, 
	\item $A$ be the linear transformation induced by~$g$ on the vector space~$(\integer_2)^3$, 
	and
	\item $f(\lambda)$ be the minimal polynomial of~$A$,
	\end{itemize}
then $f(\lambda)$ is an irreducible polynomial of degree~$3$. Since $3$ is odd, the roots of $f(\lambda)$ cannot come in pairs, so there is some root~$\alpha$ of $f(\lambda)$ (in an extension field), such that $\alpha^{-1}$ is not a root of $f(\lambda)$. Therefore $g$ and~$g^{-1}$ do not have the same minimal polynomial, so $g$ is not conjugate to $g^{-1}$ in $\GL_3(2)$. This contradicts the conclusion of the preceding paragraph.
 \end{proof}

\begin{cor} \label{112->Z7x(Z2)3xZ2}
If\/ $|G| = 112$, and $G$ has no normal Sylow\/ $7$-subgroup, then 
	$$G \iso \bigl( \integer_7 \ltimes (\integer_2)^3 \bigr) \times \integer_2 ,$$
where a generator of\/ $\integer_7$ acts via multiplication on the right by the matrix
	$$
	A = 
	\begin{bmatrix}
	0 & 1 & 0 \\
	0 & 0 & 1 \\
	1 & 1 & 0
	\end{bmatrix}
	. $$
\end{cor}

\begin{proof}
Let $P$ be a Sylow $7$-subgroup of~$G$. From \cref{112->NormalSylow}, we know that $G$ has a normal Sylow $2$-subgroup~$Q$, so $G = P \ltimes Q$. Since $P \notnormal G$, we know that $P$ acts nontrivially on~$Q$, so it also acts nontrivially on $Q/\Phi(Q)$ \cite[Thm.~5.3.5]{Gorenstein-FinGrps}. 

\setcounter{case}{0}

\begin{case}
Assume $\Phi(Q)$ is trivial.
\end{case}
Then $Q \iso (\integer_2)^4$, and a generator~$x$ of~$P$ acts by a linear transformation. 
Since 
	$$ |\GL_4(2)| = (2^4 - 1)(2^4 - 2)(2^4 - 2^2)(2^4 - 2^3) $$
is not divisible by~$7^2$, we know that all subgroups of order~$7$ in $\GL_4(2)$ are conjugate, so the semidirect product $\integer_7 \ltimes (\integer_2)^4$ is unique. Therefore $G$ must be as described.

\begin{case}
Assume $\Phi(Q)$ is nontrivial.
\end{case}
Since $7 \nmid 2^i - 1$ for $1 \in \{1,2\}$, we must have $|Q/\Phi(Q)| = 2^3$. (So $\Phi(Q) = Q'$ has order~$2$.) Therefore, a generator~$x$ of~$P$ acts transitively on the nonidentity elements of $Q/\Phi(Q)$. 

If $Q$ is nonabelian, then $Q$ is one of the groups listed in \fullcref{ListOrder16}{Z2xZ2xZ2}, since
$Q/Q'\iso (\integer_2)^3$. In each of these groups, $\Phi(Q)$ is a proper subgroup of $Z(Q)$. Thus $Z(Q)/\Phi(Q)$
is a proper subspace of $Q/\Phi(Q)$; a contradiction. 

Therefore $Q$ is abelian. Then we see that every element of~$Q$ has order~$2$, for otherwise the elements of order~$2$ in~$Q$ provide an invariant, proper subspace of $Q/\Phi(Q)$. This contradicts the fact that $\Phi(Q)$ is nontrivial.
\end{proof}

\begin{prop} \label{112NotNormal}
If\/ $|G| = 112$, and the Sylow $7$-subgroups are not normal, then every connected Cayley graph on~$G$ has a hamiltonian cycle.
\end{prop}

\begin{proof}
\Cref{112->Z7x(Z2)3xZ2} provides an explicit description of~$G$.
Let 
	\begin{itemize}
	\item $x$ be a generator of $P = \integer_7$,
	\item $v = (1,0,0) \in (\integer_2)^3$, 
	and
	\item $z$ be a generator of $Z(G) \iso \integer_2$.
	\end{itemize}

\setcounter{case}{0}

\begin{case}
Assume $\#S = 2$.
\end{case}
Let $s$ be an element of~$S$ that is not in~$Q$. Replacing $s$ by a conjugate, we may assume $s$ is either $x$ or~$xz$. 

Since $|x| = 7$, we know that the minimal polynomial of~$x$ is a divisor of 
	$$\lambda^6 + \lambda^5 +\lambda^4 + \lambda^3 + \lambda^2 + \lambda + 1 = (\lambda^3 + \lambda + 1)(\lambda^3 + \lambda^2 + 1) \quad \text{(over $\integer_2$)} . $$
So the minimal polynomial of~$x$ is either $\lambda^3 + \lambda + 1$ or $\lambda^3 + \lambda^2 + 1$.
Since (as explained in the proof of \cref{112->NormalSylow}), the minimal polynomials of $x$ and~$x^{-1}$ are not the same, we may assume (by replacing~$x$ with $x^{-1}$ if necessary) that the minimal polynomial of $x$ is $\lambda^3 + \lambda + 1$. Then (for any basis of the form $\{v, v^x, v^{x^2} \}$), $x$~acts on $(\integer_2)^3$ via multiplication on the right by the matrix~$A$ in the statement of \cref{112->Z7x(Z2)3xZ2}. (This is ``Rational Canonical Form.'')

Let $t$ be the other element of~$S$. We have $t = x^i v z^j$ for some $i$ and~$j$, where $v = (1,0,0) \in (\integer_2)^3$.
We may assume $i \in \{0,1,2,4\}$ (by replacing $t$ with its inverse if necessary). 
Consider the basis $\{a,b,c\}$ of $(\integer_2)^3$ where 
	$$ \text{$a=v$, $b=v^{x}$, and $c=b^{x}=v^{x^2}$.} $$
Then $G$ is given by
\begin{align*}
G= \left \langle a,b,c,x,z
	\mathrel{\Bigg|}
	\begin{matrix} a^2=b^2=c^2=x^7=z^2=e,\; a^b=a^c=a^z=a, \\
                            \; b^c=b^z=b, c^z=c,\; a^x=b,\; b^x=c,\; c^x=ab,\; x^z=x
          \end{matrix}
          \right\rangle
\end{align*}

Let $\psi \colon G\to  G$ be the homomorphism defined by 
	$$ \text{$\psi(a)=bc$, $\psi(b)=ac$, $\psi(c)=b$, $\psi(x)=abcxz$, and $\psi(z)=z$.} $$
One can show that the relations of the group are preserved by $\psi$ and that $\psi$ is onto. Thus $\psi$ is an automorphism of~$G$ that sends the pair $(x^4v,x)$ to $(x,x^2v)$; so we may assume $i\ne 4$.

Also, if $i = 1$, then \cref{pkSubgrp} applies. Thus, the generating sets to consider are:
	\begin{itemize}
	\item $S = \{x, vz \}$,
	\item $S = \{x, x^2vz \}$,
	\item $S = \{xz, v \}$,
	\item $S = \{xz, v z \}$,
	\item $S = \{xz, x^2v \}$,
	\item $S = \{xz, x^2v z \}$.
	\end{itemize}
In all cases, $x$ acts on $(\integer_2)^3$ via multiplication on the right by the matrix~$A$, with respect to the basis $\{a,b,c\}$.

\begin{subcase}
Assume $S = \{x, vz \}$.
\end{subcase}
A hamiltonian cycle in $\Cay(G/\langle z \rangle; S)$ is given by:
	\begin{eqnarray*}
	\bigl(x^{-1},vz,x,vz, x^4,vz, x, vz, x^2, vz,x^{-2}, vz, x^{-3}, vz,(x^{-4}, vz)^2,\ \\ x^{-3}, vz,  x, vz, x^2, vz, x^3, vz, x^{-2}, vz, (x, vz)^2, x^{-3}, vz, x^{-1}\bigr)
	\end{eqnarray*}  
To verify this, we list the vertices (according to a coset representative of $\langle z\rangle$) in the order they are visited:
\begin{eqnarray*}
&e, x^6, bx^6, b, ab, abx, abx^2, abx^3, abx^4, x^4, x^5, cx^5, cx^6, c, ac, acx^6, \\
&acx^5, ax^5, ax^4, ax^3, ax^2, bcx^2, bcx, bc,bcx^6, bcx^5, bx^5, bx^4, bx^3, bx^2, bx, abcx, \\
&abc, abcx^6, abcx^5, abx^5, abx^6, ax^6, a, ax, cx, cx^2, cx^3, cx^4, abcx^4, abcx^3, abcx^2, x^2, \\
&x^3, bcx^3, bcx^4, acx^4, acx^3, acx^2, acx, x,  e
. \end{eqnarray*}
Since $z$ is in the center of~$G$, and the generator $vz$ is used an odd number of times in the hamiltonian cycle (specifically, 17 times), the endpoint in~$G$ is~$z$, so \cref{FGL} provides a hamiltonian cycle in $\Cay(G;S)$.

\begin{subcase}
Assume $S = \{x, x^2vz\}$.
\end{subcase}
A hamiltonian cycle in $P \backslash {\Cay(G;S)}$ is:
$$\bigl(x^2vz, x^{-1}, x^2vz, x, (x^2vz, x^{-1})^2, (x^2vz)^{-2},x,x^2vz, x, (x^2vz)^3 \bigr) . $$
It passes through the vertices of the quotient multigraph in the order:
$$P, Paz, Pacz, Pab, Pbc, Pcz, Pbz, Pb, Pa, Pz, Pabc, Pac, Pabz, Pbcz, Pc, Pabcz, P . $$ 
The edge between $Pbz$ and $Pb$ is a double edge, coming from both $x^2vz$ and $(x^2vz)^{-1}$, so  \cref{MultiDouble} provides a hamiltonian cycle in  $\Cay(G;S)$. 

\begin{subcase}
Assume $S = \{xz, v\}$.
\end{subcase}
A hamiltonian cycle in $P \backslash {\Cay(G;S)}$ is:
$$\bigl(xz, v, (xz)^4, v, xz, v, (xz)^{-1}, v, xz, v, (xz)^2, v \bigr) . $$
It passes through the vertices of the quotient multigraph in the order:
$$P, Pz, Paz, Pb, Pcz, Pab, Pbcz, Pabcz, Pac, Pc, Pbz, Pabz, Pbc, Pabc, Pacz, Pa, P .$$ 
The edge between $P$ and $Pz$ is a double edge, coming from both $xz$ and $(xz)^{-1}$, 
so  \cref{MultiDouble} provides a hamiltonian cycle in  $\Cay(G;S)$. 

\begin{subcase}
Assume $S = \{xz, vz\}$.
\end{subcase}
A hamiltonian cycle in $P \backslash {\Cay(G;S)}$ is:
$$\bigl(xz, vz, (xz)^2, vz, (xz)^{-1}, vz, (xz)^{-2}, vz, ((xz)^{-1}, vz)^3 \bigr) . $$
It passes through the vertices of the quotient multigraph in the order:
$$P, Pz, Pa, Pbz, Pc, Pacz, Pabc, Pbcz, Pab, Pcz, Pac, Pabcz, Pbc, Pabz, Pb, Paz, P .$$ 
The edge between $P$ and $Pz$ is a double edge, coming from both $xz$ and $(xz)^{-1}$.

\begin{subcase}
Assume $S = \{xz, x^2v\}$.
\end{subcase}
A hamiltonian cycle in $P \backslash {\Cay(G;S)}$ is:
$$\bigl(xz, x^2v, (xz)^2, (x^2v, xz)^2, (xz)^2, x^2v, xz, x^2v, (xz)^{-2}, (x^2v)^{-1} \bigr) . $$
It passes through the vertices of the quotient multigraph in the order:
$$P, Pz, Paz, Pb, Pcz, Pabcz, Pac, Pab, Pbcz, Pabc, Pacz, Pabz, Pbc, Pc, Pbz, Pa, P . $$ 
The edge between $P$ and $Pz$ is a double edge, coming from both $xz$ and $(xz)^{-1}$, so  \cref{MultiDouble} provides a hamiltonian cycle in  $\Cay(G;S)$. 

\begin{subcase}
Assume $S = \{xz, x^2vz\}$.
\end{subcase}
A hamiltonian cycle in $P \backslash {\Cay(G;S)}$ is:
$$\bigl((xz, x^2vz)^5, (xz)^5, (x^2vz)^{-1} \bigr) . $$
It passes through the vertices of the quotient multigraph in the order:
$$P, Pz, Pa, Pbz, Pb, Pcz, Pabc, Pacz, Pab, Pbcz, Pc, Pabz, Pbc, Pabcz, Pac, Paz, P . $$ 
The edge between $P$ and $Pz$ is a double edge, coming from both $xz$ and $(xz)^{-1}$, so  \cref{MultiDouble} provides a hamiltonian cycle in  $\Cay(G;S)$.

\begin{case}
Assume $\#S > 2$.
\end{case}
Every minimal generating set of $G/Z(G) \iso \integer_7 \ltimes (\integer_2)^3$ has only $2$ elements, so there exist $s,t \in S$, such that $\langle s, t \rangle = \langle x, v \rangle$. We may assume $s = x$. And we have $t = x^i v$.

Since $\langle s,t \rangle$ has index~$2$ in~$G$, we must have $\#S = 3$; let $u$~be the third element of~$S$, so $u= x^j w z$ with $w \in (\integer_2)^3$. 
	\begin{itemize}
	\item Since $\langle x ,  x^j w z \rangle = \langle s , u \rangle \neq G$, we must have $w = e$. So $u = x^j z$.
	\item Then we must have $j = 0$, for otherwise $\langle u \rangle = \langle x, z \rangle \ni x = s$, which contradicts the fact that $S$ is a minimal generating set. 
	\end{itemize}
But then $u = z \in Z(G)$, so \fullcref{NormalEasy}{Z} and  \cref{QuotientIs4pOr8p} apply.
\end{proof}

\section{Preliminaries on groups of order~$16p$} \label{prelim-16p}

Because of \cref{pNotNormal}, we henceforth 
	$$ \text{\bf assume that the Sylow $\boldsymbol{p}$-subgroups of~$\boldsymbol{G}$ are normal} .$$

\begin{notation}
Throughout the remainder of this paper:
	\begin{itemize}
	\item $G$ is a group of order~$16p$, where $p$ is an odd prime,
	\item $P \iso \integer_p$ is a Sylow $p$-subgroup of~$G$ (and $P \normal G$),
	\item $Q$ is a Sylow $2$-subgroup of~$G$, so $|Q| = 16$,
	and
	\item $S$ is a minimal generating set of~$G$.
	\end{itemize}
We wish to show $\Cay(G;S)$ has a hamiltonian cycle.
\end{notation}

We know $P \iso \integer_p$, and the possibilities for~$Q$ are given in \cref{ListOrder16}.

\begin{lem} \label{GLem} We may assume:
	\begin{enumerate}
	\item \label{GLem-QxP}
	$G = Q \ltimes P$,
	\item \label{GLem-Qnonabel}
	$Q$ is nonabelian, and acts nontrivially on~$P$,
	\item \label{GLem-G'=Q'xP}
	$G' = Q' \times P$ is cyclic,
	\end{enumerate}
\end{lem}

\begin{proof}
\pref{GLem-QxP} Since $P\normal G$, $G=QP$, and $Q\cap P=\{e\}$, we have $G\iso Q\ltimes P$.

(\ref{GLem-Qnonabel}, \ref{GLem-G'=Q'xP})
Since the automorphism group of $P \iso \integer_p$ is abelian, we know $Q'$ centralizes~$P$. So \pref{GLem-QxP} implies that $Q' \le G' \le Q' \times P$. Since all of the groups in \cref{ListOrder16} have a cyclic commutator subgroup, we know $Q'$ is cyclic.
Also, by \cref{KeatingWitte}, we may assume $G'$ is not a cyclic subgroup of prime-power order.
 Thus, we may assume $G' \neq Q'$ and $G' \neq P$. So $G' = Q' \times P$ (and $Q' \neq \{e\}$, so $Q$ is nonabelian).  Since $Q'$ and~$P$ are both cyclic, this implies $G'$~is cyclic. Furthermore, since $P \subset G'$, we know that $G \not\iso Q \times P$, so the action of~$Q$ on~$P$ is nontrivial.
\end{proof}

The following \lcnamecref{G/G'} shows there are three possibilities for $G/G'$; each of these possibilities will be considered individually, in \cref{Z2xZ2Sect,,Z4xZ2Sect,Z2xZ2xZ2Sect}, respectively.

\begin{cor} \label{G/G'}
We may assume $G/G'$ is isomorphic to either $\integer_2 \times \integer_2$, $\integer_4 \times \integer_2$, or $\integer_2 \times \integer_2 \times \integer_2$.
\end{cor}

\begin{proof}
Since $G = QP$ and $G' = Q' P$ \csee{GLem}, we have $G/G' \iso Q/Q'$. Then the desired conclusion follows from inspection of \cref{ListOrder16}.
\end{proof}

\begin{cor} \label{Q'Cor}
We may assume
	\begin{enumerate}
	\item  \label{Q'Cor-normal}
	$Q' \normal G$,
	\item \label{Q'Cor-Phi}
	 $Q' \le \Phi(G)$, 
	and
	\item  \label{Q'Cor-Smin}
	$S$ is a minimal generating set of $G/Q'$.
	\end{enumerate}
\end{cor}

\begin{proof}
\pref{Q'Cor-normal} From \fullcref{GLem}{G'=Q'xP}, we know that $Q'$ is normalized by~$P$ (indeed, it is centralized by~$P$). Then, since it is also normalized by~$Q$, it is normalized by $PQ = G$.

\pref{Q'Cor-Phi} Let $M$ be a maximal subgroup of~$G$. 
	\begin{itemize}

	\item If $M$ contains~$P$, then $M/P$ is a maximal subgroup of $G/P$, so $M/P$ contains $\Phi(G/P) = \Phi(Q) P/P \ge Q' P/P$, so $M \ge Q'$. 

	\item If $M$ does not contain~$P$, then $M$~is a $2$-group, so the maximality implies it is a Sylow $2$-subgroup of~$G$. Every Sylow $2$-subgroup (such as~$M$) contains every normal $2$-subgroup (such as~$Q'$), so $M \ge Q'$.

	\end{itemize}
Thus, every maximal subgroup of~$G$ contains $Q'$, so $Q' \le \Phi(G)$.

\pref{Q'Cor-Smin} Since $S$ is a minimal generating set of~$G$, this follows from \pref{Q'Cor-Phi}.
\end{proof}

\begin{cor} \label{G'2gen}
If $G = \langle a , b \rangle$ is $2$-generated, then $G' = \langle [a,b] \rangle$.
\end{cor}

\begin{proof}
By \fullcref{GLem}{G'=Q'xP}, $G'$~is cyclic. Since every subgroup of a cyclic, normal subgroup is normal, we know $\langle [a,b] \rangle \normal G$. Since $\langle a,b \rangle = G$, and $a$~commutes with~$b$ in $G/\langle [a,b] \rangle$, we know $G/\langle [a,b] \rangle$ is abelian, so $G' \subset \langle [a,b] \rangle$. The opposite inclusion is obvious.
\end{proof}

\section{\texorpdfstring{The case where $G/G' \iso \integer_2 \times \integer_2$}{The case where G/G' = Z2 x Z2}} \label{Z2xZ2Sect}

\begin{prop}
Assume $|G| = 16p$.
If $G/G' \iso \integer_2 \times \integer_2$, then $\Cay(G;S)$ has a hamiltonian cycle.
\end{prop}

\begin{proof} We proceed via case-by-case analysis.

\setcounter{case}{0}

\begin{case}
Assume $\#S = 2$.
\end{case}
Write $S = \{a,b\}$. Then $(a^{-1},b^{-1},a,b)$ is a hamiltonian cycle in $\Cay(G/G';S)$ whose endpoint in~$G$ is $a^{-1} b^{-1} a b = [a,b]$. This generates~$G'$ \csee{G'2gen}, so \cref{FGL} applies.  

\begin{case}
Assume $\#S \ge 3$.
\end{case}
Since $|G/Q'| = 4p$ is a product of only $3$ primes, \fullcref{Q'Cor}{Smin} implies $\#S \le 3$. Therefore $\#S = 3$; write 
	$$S = \{a,b,c\} .$$

\begin{subcase}
Assume $|c|$ is divisible by~$p$.
\end{subcase}
Since $|G/Q'|$  is a product of only $3$ primes, and $P$ is the unique subgroup of order~$p$ in~$G$, the minimality of~$S$ (and \fullcref{Q'Cor}{Smin}) implies 
	\begin{itemize}
	\item the image of $\langle c \rangle$ in $G/Q'$ has order~$p$,
	and
	\item the image of $\langle a,b \rangle$ in $G/Q'$ has order~$4$.
	\end{itemize}
Thus, $b$ has order~$2$ in $G/Q'$, so $b$ either centralizes $P$ or inverts it: let $\epsilon\in \{\pm 1\}$,
such that $w^b = w^{\epsilon}$ for all $w\in P$. Since $Q'$ is a cyclic group of order~$4$, its only automorphisms are the identity automorphism and the one that inverts every element in~$Q'$. Let $\epsilon'\in  \{\pm 1\}$, such that $u^b=u^{\epsilon'}$ for all $u\in Q'$.
Write $c=uw$ for some $u\in Q'$ and $w\in P$. Then
	$$ \text{$c^b  = u^{\epsilon'} w^{\epsilon} = c^{\epsilon} u^{\epsilon'-\epsilon} \in c^{\epsilon}\Phi(Q')$, \  since \ $\epsilon'-\epsilon\in\{0,\pm 2\}$.} $$
Now 
	$$ ( a^{-1} , c^{-(p-1)} , b^{-1}, c^{\varepsilon(p-1)}, a , c^{-(p-1)} , b, c^{\varepsilon(p-1)} ) $$
is a hamiltonian cycle in $\Cay\bigl( G/Q' ; S \bigr)$
 whose endpoint in $G/\Phi(Q')$ is
	\begin{align*}
	a^{-1}  c^{-(p-1)}  b^{-1} c^{\varepsilon(p-1)} a  c^{-(p-1)}  b c^{\varepsilon(p-1)} 
	= a^{-1}   b^{-1} a   b 
	= [a,b] 
	, \end{align*}
which generates $Q'$ \csee{G'2gen}. So \cref{FGL} provides a hamiltonian cycle in $\Cay(G;S)$.

\begin{subcase}
Assume no element of~$S$ has order divisible by~$p$.
\end{subcase}
This implies that every element of~$S$ is a $2$-element. Also, since $Q / Q'$ is a Sylow $2$-subgroup of $G/Q'$, and $Q/Q' \iso G/G' \iso \integer_2 \times \integer_2$, we know that $G/Q'$ has no elements of order~$4$. Therefore
	$$ \text{every element of~$S$ has order~$2$ in $G/Q'$.} $$
So we may assume 
	every element of~$S$ has order~$2$ in $G/\Phi(Q')$,
for otherwise \cref{FGL(order2)} and  \cref{QuotientIs4pOr8p} apply with $N = Q'$.
Then we may assume 
	$$ \text{every element of~$S$ has order~$2$} , $$
for otherwise \cref{FGL(order2)} and  \cref{QuotientIs4pOr8p} apply with $N = \Phi(Q')$.

Since $G/P \iso Q$ is a $2$-generated $2$-group, we know that all of its minimal generating sets have the same cardinality, so some $2$-element subset of~$S$ generates $G/P$. Since two elements of order~$2$ always generate a dihedral group, we conclude that 
	$$ Q \iso D_{16}=\langle f,t| f^2=t^8=(ft)^2=e\rangle .$$

\begin{subsubcase} \label{NoSCentP}
Assume no element of~$S$ centralizes~$P$.
\end{subsubcase}
Let $\overline{S}$ be the image of~$S$ in 
	$$ G/Q' \iso (\integer_2 \times \integer_2) \ltimes \integer_p \iso D_{4p} .$$
From \fullcref{Q'Cor}{Smin}, we see that $\overline{S}$ is a minimal generating set of~$D_{4p}$.
Also, by the assumption of this \lcnamecref{NoSCentP},
we know that every element of~$\overline{S}$ is a reflection; let $f \in \overline{S}$. 
There are only two proper subgroups of $D_{4p}$ that properly contain $\langle f \rangle$ (because $\integer_2$ and $\integer_p$ are the only nontrivial proper subgroups of the group $\integer_{2p}$ of rotations), so we may assume $\overline{S} = \{f, fx, fy\}$, where $x$ and~$y$ are rotations of order~$2$ and~$p$ in~$D_{4p}$, respectively. Then $\langle fx, fy \rangle = D_{4p}$, which contradicts the minimality of~$\overline{S}$. 

\begin{subsubcase} \label{SHasCentP}
Assume $S$ contains an element that centralizes~$P$.
\end{subsubcase}
Each element of~$S$ must map to a reflection in $G/P \iso Q \iso D_{16}$ (since the elements of~$S$ all have order~$2$ in both $G/P$ and $G/(Q'P)$). Then, by the assumption of this \lcnamecref{SHasCentP}, we know that some reflection centralizes~$P$. Because $Q$ acts nontrivially on $P$, we have
	$$ G = \langle f, t, w \mid f^2 = t^8= w^p = e, \ ftf = t^{-1}, \ fwf = w, \ t^{-1} w t = w^{-1} \rangle .$$
From the assumption of this \lcnamecref{SHasCentP}, we may assume $f \in S$. By the minimality of~$S$, we must have $\langle f, s \rangle = Q$, for some $s \in S$ (after replacing~$Q$ by a conjugate). Since all elements of $S \cap Q$ are reflections, we may assume $s = ft$. To generate $G$ (and map to a reflection in $G/P$), the final element of~$S$ must be of the form $ft^i w^j$, with $p \nmid j$. By replacing $w$ with~$w^j$, we may assume $j = 1$. So the final element of~$S$ is $f t^i w$. Since all elements of~$S$ have order~$2$ in~$G$, it must be the case that $t^i$ inverts~$w$, so $i$~is odd. 
So
	$$ \text{$S$ is of the form $\{f, ft, ft^iw \}$, with $i$ odd.} $$

By \fullcref{Q'Cor}{normal}, we know $Q' \normal G$. 
Observe that the image of~$f$ is central in $G/Q'$, 
and $ft^iw \equiv ftw \pmod{Q'}$ (because $Q' = \langle t^2 \rangle$ and $i$~is odd), so
$$\Cay(G/Q';S)\iso \Cay \bigl( \langle ft,ft^i w\rangle/Q'; 
\{ft,ft^i w\} \bigr)\times \Cay \bigl( \langle f\rangle;\{ f\} \bigr)\iso C_{2p}\times K_2$$
is a prism, which has the natural hamiltonian cycle $\bigl( (ft,ft^i w)^p \#,f \bigr)^2$.
The endpoint in $G$ is 
\begin{align*}
\Bigl( \bigl( (ft)(ft^i w) \bigr)^{p}(ft^i w)^{-1}f \Bigr)^2 &= \Bigl( (t^{i-1}w)^p w^{-1} t^{-i} \Bigr)^2 \\
&= \bigl( t^{(i-1)p-i} w^{-1} \bigr)^2 \\
&= t^{2(i-1)p-2i}.
\end{align*}
Since $i$~is odd, the exponent of~$t$ is congruent to~$2$ modulo~$4$, so the endpoint generates $\langle t^2 \rangle =Q'$. 
Thus, \cref{FGL} provides a hamiltonian cycle in
$\Cay(G;S)$.
\end{proof}

\section{\texorpdfstring{The case where $G/G' \iso \integer_4 \times \integer_2$}{The case where G/G' = Z4 x Z2}} \label{Z4xZ2Sect}

\begin{prop}
Assume $|G| = 16p$.
If $G/G' \iso \integer_4 \times \integer_2 $, then $\Cay(G;S)$ has a hamiltonian cycle.
\end{prop}

\begin{proof} 
We proceed via case-by-case analysis.

\setcounter{case}{0}

\begin{case}
Assume $\#S = 2$.
\end{case}
%Let $S = \{a,b\}$, with $a$~of order~$4$ in $G/G'$. Then $(a^{-3}, b^{-1}, a^3, b)$ is a hamiltonian cycle in $G/G'$. Let us look at its endpoint $[a^3, b]$ in $G/Q'$ and $G/P$:
%	\begin{itemize}
%	\item (mod $P$) Since $G/P \iso Q$ has class~$2$, we have
%		$$ [a^3, b] \equiv [a,b]^3 \equiv [a,b] \pmod{P} ,$$
%	so it generates $Q' \pmod{P}$.
%	\item (mod $Q'$) 
%Let $k \in \integer$ with $x^a = x^k$ for $x \in P$. Then
%		$$ [a^3, b] 
%		\equiv [a,b]^{a^2} [a,b]^{a} [a,b]
%		=  [a,b]^{k^2} [a,b]^{k} [a,b]
%		=  [a,b]^{k^2 + k + 1}
%		\pmod{Q'} .$$
%	This generates $G'$ unless $\gcd(k^2 + k + 1, 2p) > 1$. 
%	\end{itemize}
Let 
	\begin{itemize}
	\item $S = \{a,b\}$, with $a$~of order~$4$ in $G/G'$,
	and
	\item $k \in \integer$ with $g^a = g^k$ for $g \in G'$. 
	\end{itemize}
Then $(a^{-3}, b^{-1}, a^3, b)$ is a hamiltonian cycle in $G/G'$, and its endpoint  in $G$ is
	$$ [a^3, b] 
		= [a,b]^{a^2} [a,b]^{a} [a,b]
		=  [a,b]^{k^2} [a,b]^{k} [a,b]
		=  [a,b]^{k^2 + k + 1}
		 .$$
By \cref{G'2gen}, this generates $G'$ unless $\gcd(k^2 + k + 1, |G'|) > 1$. Since $|G'| = 2p$, and $k^2 + k + 1$ is always odd, this generates $G'$ unless
	$ k^2 + k + 1 \equiv 0 \pmod{p}$, which implies $k^3 \equiv 1 \pmod{p}$. This means that $a^3$ centralizes~$P$. But $a^4 \in G' \le C_G(P)$, so this implies that $a$~centralizes~$P$: therefore $k \equiv 1 \pmod{p}$. Since $ k^2 + k + 1 \equiv 0 \pmod{p}$, we conclude that 
	$$ \text{$p = 3$ and $a$ centralizes~$P$.} $$

From \fullcref{GLem}{Qnonabel} (and the fact that $a$ centralizes~$P$), we know that 
	$$ \text{$b$ does not centralize~$P$.} $$
Since $G'=Q' P$ is cyclic of order~$6$, we know $G'$ has only two automorphisms; namely, the identity
automorphism and the automorphism that inverts~$G'$. Thus $x^b=x^{-1}$ for all $x\in G'$.
If $b$ has order 4 in $G/G'$, then the hamiltonian cycle $(b^{-3},a^{-1},b^3,a)$ in $\Cay(G/G';S)$ has endpoint
	$$ [b^3,a]=[b,a]^{b^2} [b,a]^{b} [b,a] = [b,a] [b,a]^{-1} [b,a]= [b,a] $$
in~$G$, which generates~$G'$. Thus $\Cay(G;S)$ has a hamiltonian cycle by \cref{FGL}. So we may assume that $b$ has order~$2$ in $G/G'$.
Write $b=qw$ for some $q\in Q$ and $w\in P$, where $q^{-1}wq=w^{-1}$. Then $b^2=q^2w^q  w=q^2$.
Hence, the order of~$b$ is not divisible by~$p$, so $b$ is a $2$-element.
Thus, we may assume (after replacing $Q$ by a conjugate) that $b \in Q$. Thus, the order of~$b$ in $G/Q'$ is~$2$. So we may assume $|b| = 2$, for otherwise \cref{FGL(order2)} and  \cref{QuotientIs4pOr8p} apply.

Since $b \in Q$, and $\langle a,b \rangle = G$, we know $a \notin Q$. Since $a$~centralizes~$P$, this implies that $|a|$~is divisible by~$p$ (i.e.,~$3$). But $|a|$ is also a multiple of~$4$ (its order in $G/G'$). So $|a|$ is divisible by~$12$. Since $|G| = 16p = 48$ (and~$G$ is not cyclic), this implies $|a|$ is either $12$ or~$24$. If $|a| = 24$, then $G = \langle b \rangle \ltimes \langle a \rangle$, so \cref{AlspachSemiProd} applies. So we may assume $|a| = 12$.

Now, of the 3 groups listed in \fullcref{ListOrder16}{Z4xZ2}, 
	\begin{itemize}
	\item group~\pref{ListOrder16-Z4xZ2-Z2xZ8} has no generating set without an element of order~$8$, 
	and
	\item group~\pref{ListOrder16-Z4xZ2-Z4xZ4} has no $2$-element generating set with an element of order~$2$.
	\end{itemize}
So $Q$ must be group~\pref{ListOrder16-Z4xZ2-Z4x(Z2xZ2)},
and we may assume $a = xw$ (by relabeling the elements of~$Q$).

Since $\langle a,b \rangle = G$, we know $\langle x, b \rangle = Q$, so (since $|b| = 2$), we may assume $b = y$ (by further relabeling the elements of~$Q$).
Note that, since $z\in Z(G)$, we know $z$~centralizes both $a$ and~$b$. 

Let $N=\langle  a^2 \rangle = \langle x^2, w \rangle$, and consider the hamiltonian cycle $\bigl( (b,a)^4\#,a^{-1} \bigr)$ in $\Cay(G/N;S)$, which passes through the vertices in the following order:
	$$ N, Ny, Nxyz, Nxz, Nz, Nyz, Nxy, Nx, N. $$
Its endpoint in~$G$ is 
	$$(ba)^4a^{-2} = (yxw)^4 a^{-2} = e \cdot a^{-2} =a^{-2} , $$
which generates $\langle a^2\rangle = N$. So \cref{FGL} provides a hamiltonian cycle in $\Cay(G;S)$.

\begin{case}
Assume $\#S > 2$.
\end{case}
Since $G/G' \iso \integer_4 \times \integer_2$, there exists a $2$-element subset $\{a,b\}$ of~$S$ that generates $G/P$. Since $\{a,b\} \subsetneq S$, and $S$ is minimal, we have $P \not\subset \langle a, b \rangle$. Therefore, by passing to a conjugate, we may assume $\langle a, b \rangle = Q$.

Let $c$~be a third element of~$S$ (so $c \notin Q$). Then $\langle a,b,c\rangle$ properly contains~$Q$. But $Q$ is a maximal subgroup of~$G$ (since $|G/Q| = p$ is prime), so this implies $\langle a,b,c \rangle = G$. So the minimality of~$S$ implies $S = \{a,b,c\}$.

\begin{claim}
We may assume, for each $s \in S$, that either $s^2 \in P$, or $sP \in \Phi(Q) P$, or $s$~acts on~$P$ via an automorphism of order~$4$.
\end{claim}
Suppose there exists $s \in S$ that has none of the three properties.
 Since $\#S > 2$ and $sP \notin \Phi(Q) P$, we know $p \nmid |s|$, so (up to conjugacy) $s \in Q$. 
Then, since $Q/Z(Q) \iso \integer_2 \times \integer_2$, we have $s^2 \in Z(Q)$, so $\langle s^2 \rangle \normal Q$. Also, since $s$~does not act on~$P$ by an automorphism of order~$4$, we know $s^2$ centralizes~$P$. Therefore $s^2 \in Z(G)$, so $\langle s^2 \rangle \normal G$, so \cref{FGL(order2)} and  \cref{QuotientIs4pOr8p} apply.

\begin{subcase}
Assume $Q \iso \integer_4 \ltimes \integer_4 = \langle x \rangle \ltimes \langle y \rangle$.
\end{subcase}
Since some element of~$S$ must generate $G/ ( \langle y \rangle P )$, we may assume $x \in S$ (after replacing $Q$ by a conjugate). That is, we may assume $a = x$.

Observe that $x^2\not\in P$ and $xP\not\in \Phi(Q)P$. Thus, the Claim % is it still a Claim ???
tells us that $x$ acts on $P$ via an automorphism of order~$4$. Hence, $Q/C_Q(P) \iso \integer_4$. Therefore, $C_Q(P)$ is a cyclic normal subgroup of~$Q$ with a cyclic quotient, so we may assume $C_Q(P) = \langle y \rangle$. 

Since $Q$ has no $2$-element generating set that contains an element of order~$2$, we know $|b| > 2$, so, from the Claim, % is it still a Claim ???
we know that $b$ generates $Q/C_Q(P) = Q/\langle y \rangle$, so $b \equiv a^{\pm1} \pmod{\langle y \rangle}$; assume (by replacing $b$ with its inverse if necessary) that $b a \in \langle y \rangle$. Then, since $\langle a,b \rangle = Q$, we must have $\langle ba \rangle = \langle y \rangle \normal G$. Then, since $|y| = 4$, \cref{FGL(gen-twice)} and  \cref{QuotientIs4pOr8p} apply.

\begin{subcase}
Assume $c \notin \Phi(Q)\, P$.
\end{subcase}
Since $Q$ does not centralize~$P$, we may assume $b$~does not centralize~$P$. 
Write $c=wu$ for some $w\in P$ and $u\in Q$.
Since $c\not\in Q$, we have $w\ne e$. Since $b$ does not centralize $P$, we have $(w^{-1})^b \ne w^{-1}$.
Thus 
	$$[b,c]=(u^{-1})^b (w^{-1})^b wu=(u^{-1})^b u ((w^{-1})^b w)^u , $$
where $(u^{-1})^b u\in Q$,
$((w^{-1})^b w)^u\in P$, and $((w^{-1})^b w)^u\ne e$. Hence, $[b,c]\not\in Q'$ and $[b,c]\in G'=Q'\times P$.
Thus the order of $[b,c]$ is a divisor of $|G'|=2p$, but not a divisor of $|Q'|=2$.
Hence, $p$ divides the order of $[b,c]$ and we must have $P\subset \langle [b,c]\rangle$.
If $c\in \{a, ab\}\Phi(Q)P$, then $\langle b,c\rangle=G$ which contradicts the minimality of $S$. Thus $c\in b\, \Phi(Q)P$. (But we may assume $c \notin b^{\pm1} P$, for otherwise \cref{FGL(gen-twice)} and  \cref{QuotientIs4pOr8p} apply.)
Since $c\not\in a\Phi(Q)P$, the argument of the preceding paragraph (by interchanging $a$ and~$b$) implies that $a$ must centralize $P$.
That is, $a \in C_Q(P)$. 
Therefore, the Claim % is it still a claim???
implies $|a| = 2$.
In summary, we know:
	\begin{itemize}
	\item $|a| = 2$, and $a$ centralizes~$P$,
	\item $b$ does not centralize~$P$,
	and
	\item $c \in b \,  \Phi(Q)\, P$, but $c \notin b^{\pm1} P$.
	\end{itemize}

\begin{subsubcase}
Assume $Q \iso \integer_2 \ltimes \integer_8 = \langle x \rangle \ltimes \langle y \rangle$.
\end{subsubcase}
Since $|a| = 2$, we may assume $a = x$. Then we must have $|b| = 8$, so we may assume $b = y$. 
Then $\Phi(Q) = \langle b^2 \rangle$, so $c \in \{b^3, b^5\} P$. By replacing $c$ with its inverse, we may assume $c \in b^5 P$. Then
	$$ \text{$S = \{a,b,c\}$ with $a = x$, $b = y$, and $c = y^5 w$, where $w$ generates~$P$}. $$
Also, from the above properties of $a$ and~$b$, we know $x$ centralizes~$P$, but $y$ acts on $P$ via an automorphism of order~$4$.

Consider the hamiltonian cycle
	$$ (b^2, c^{-1}, b^2, c, b^{-1}, a, b^{-7}, a ) $$
%	$$ \bigl( y^2, (y^5w)^{-1}, y^2, y^5w, y^{-1}, x, y^{-7}, x \bigr) $$
in $\Cay(G/P;S)$, which passes through the vertices of the graph in the order:
	$$ P, Py, Py^2, Py^5, Py^6, Py^7, Py^4, Py^3, Pxy^7, Pxy^6, Pxy^5, Pxy^4, Pxy^3, Pxy^2, Pxy, Px, P. $$
Note that, since the action of~$y$ on~$P$ has order~$4$, we know that $y^2$ inverts~$w$, so the endpoint in~$G$ is
	$$ b^2 c^{-1} b^2 c b^{-1} a b^{-7} a
	= y^2 (w^{-1} y^{-5}) y^2 (y^5 w) y^{-1} x y^{-7} x
	= y^2w^{-1}y^2 wy^{-1} xy x=w^2 , $$
which generates $P=\langle w\rangle$.
By \cref{FGL}, we have a hamiltonian cycle in $\Cay(G;S)$.

\begin{subsubcase}
Assume $Q \iso \integer_4 \ltimes (\integer_2  \times \integer_2 ) = \langle x \rangle \ltimes \langle y,z \rangle$.
\end{subsubcase}
Since $|a| = 2$, we may assume $a = y$. Then we may assume $b = x$. Since $\Phi(Q) = \langle b^2, z \rangle$, we must have $c \in \{bz, b^{-1} z \} P$. By replacing $c$ with its inverse, we may assume $c \in b z P$. Then
	$$ \text{$S = \{x, y, x z w\}$, where $w$ generates~$P$}. $$
And $y$ centralizes~$P$, but $x$ acts on $P$ via an automorphism of order~$4$.

Let $k\in \integer_p$, such that $x^{-1}wx=w^k$.
Since the action of $x$ on $P$ has order~$4$, we have $2\le k\le p-2$.
Consider the hamiltonian cycle
	$$ (xzw, x^3, y, x^2, xzw, x^3, xzw, y, x^{-3}) $$
in $\Cay(G/P;S)$ that passes through the vertices of this graph in the order:
	$$ P, Pxz, Px^2z, Px^3z, Pz, Pyz, Pxy, Px^2yz, Px^3yz, Py, Pxyz, Px^2y, Px^3y, Px^3, Px^2, Px, P. $$
The endpoint of this cycle in~$G$ is
\begin{align*}
(xzw)(x^3 yx^3 zw)(x^4 zw)(yx^{-3}) &= (xzw)(x^2 yw)(zw)(xyz) \\
&= x^4 y^2 z^4 w^{-k} w^k w^k = w^k.
\end{align*}
Since $k$ is coprime to $p$ (recall $2\le k\le p-2$), this generates $P=\langle w\rangle$.
By \cref{FGL}, we have a hamiltonian cycle in $\Cay(G;S)$.

\begin{subcase}
Assume $c \in \Phi(Q)\, P$.
\end{subcase}
We may assume $c \notin G'$, for otherwise \cref{CyclicNormal2p} applies.

\begin{subsubcase}
Assume $Q \iso \integer_2 \ltimes \integer_8 = \langle x \rangle \ltimes \langle y \rangle$.
\end{subsubcase}
Up to automorphism, any $2$-element generating set of~$Q$ is of the form $\{x y^i, y\}$. 
Since $xy^4$ has order~$2$, it may be replaced with~$x$ (if $i \in \{3,4,5\}$). This implies that we may assume $-2 \le i \le 2$.  Then, since we may replace $y$ with~$y^{-1}$, we may assume $0 \le i \le 2$. However, $xy^2$ has order~$4$, but its square is in~$Q'$, so it cannot act on~$P$ by an automorphism of order~$4$; therefore, the Claim %
implies it is not in~$S$. So $i \in \{0,1\}$.

Since $\Phi(Q) = \langle y^2 \rangle$ and $Q' = \langle y^4 \rangle$, we must have $c \in \{y^2, y^6\}P$; replacing $c$ with~$c^{-1}$, we may assume $c \in y^2 P$.  Thus, either
	$$ \text{$S = \{x, y, y^2 w\}$ \quad or \quad $S = \{xy, y, y^2 w\}$, \quad where $\langle w \rangle = P$} .$$
Also:
	\begin{itemize}
	\item $x$ either centralizes~$P$ or inverts it,
	and
	\item $y$ acts on~$P$ by an automorphism of order~$4$.
	\end{itemize}

Let $\epsilon\in\{\pm 1 \}$, such that
$x^{-1} wx=w^{\epsilon}$, and let $k\in \integer_p$, such that $y^{-1}wy=w^k$.
Since the action of $y$ on $P$ has order~$4$, we have $2\le k\le p-2$.
Let $N=\langle y^4, P\rangle=\langle y^4 w\rangle\iso \integer_{2p}$.

For the first generating set,
	$$ (y^2w, y^{-1}, y^2w, x, y^{-3}, x) $$
is a hamiltonian cycle
in $\Cay(G/N;S)$ that passes through the vertices of this graph in the order:
	$$ N, Ny^2, Ny, Ny^3, Nxy^3, Nxy^2, Nxy, Nx, N. $$
The endpoint of this cycle in~$G$ is
\begin{align*}
y^2 w yw xy^{-3}x =y^2wywy= y^2 w y^2 w^k = y^4 w^{k-1}.
\end{align*}
Since $k-1$ is coprime to~$p$ (recall $2\le k\le p-2$), and $y^4$ has order 2, this generates $N=\langle y^4 w\rangle$.
So \cref{FGL} provides a hamiltonian cycle in $\Cay(G;S)$.

For the second generating set,
	$$ \bigl( xy, y^{-1}, y^2w, y, (xy)^{-1}, y^{-1}, y^{2}w, y \bigr) $$
is a hamiltonian cycle
in $\Cay(G/N;S)$ that passes through the vertices of this graph in the order:
	$$ N, Nxy, Nx, Nxy^2, Nxy^3, Ny^2, Ny, Ny^3, N. $$
The endpoint of this cycle in $G$ is
\begin{align*}
xy^2 w x^{-1} ywy =xy^2 x^{-1} w^{\epsilon} y^2 w^k = y^4 w^{k-\epsilon}
\end{align*}
(with $\epsilon\in\{\pm 1\}$ depending on whether $x$ centralizes  or inverts~$P$.)
Since $k-1$  and $k+1$ are coprime to~$p$ (recall $2\le k\le p-2$) and $y^4$ has order~$2$, this generates $N=\langle y^4 w\rangle$.
So \cref{FGL} provides a hamiltonian cycle in $\Cay(G;S)$.

\begin{subsubcase}
Assume $Q \iso \integer_4 \ltimes (\integer_2  \times \integer_2 )$.
\end{subsubcase}
Up to automorphism, any $2$-element generating set of~$Q$ is of the form $\{x, x^i y\}$. Of course, by replacing $x$ with~$x^{-1}$, we may assume $0 \le i \le 2$. Also, since $x^2 y$ has order~$2$, it may be replaced with~$y$; so we may assume $i \in \{0,1\}$. 
 
Since $\Phi(Q) = \langle x^2 , z \rangle$ and $Q' = \langle z \rangle$, we must have $c \in \{x^2, x^2 z\}P$.  Thus, letting $w$ be a generator of~$P$, either:
	$$ \text{$S = \{x, y, x^2 w\}$ 
	\  or \  $S = \{x, xy, x^2 w\}$,
	\  or \  $S = \{x, y, x^2 z w\}$,
	\  or \  $S = \{x, xy, x^2 z w\}$%
	} . $$
Also:
	\begin{itemize}
	\item $x$ acts on~$P$ by an automorphism of order~$4$ (so $x^2$ inverts~$P$),
	\item $y$ either centralizes~$P$ or inverts it,
	and
	\item $z$ centralizes~$P$.
	\end{itemize}

Let $k\in \integer_p$, such that $x^{-1}wx=w^k$.
Since the action of $x$ on $P$ has order~$4$, we have $2\le k\le p-2$.
Let $N=\langle z, P\rangle$. Since $z$ centralizes $P$, the order of $z$ is~$2$, and the order of $w$ is~$p$, we have $N=\langle zw\rangle \iso \integer_{2p}$.

For the first generating set,
	$$ \bigl( x^2w, x^{-1}, x^2w, y, x^{-3},y \bigr) $$
is a hamiltonian cycle in $\Cay(G/N;S)$ that passes through the vertices of this graph in the order:
	$$ N, Nx^2, Nx, Nx^3, Nx^3y, Nx^2y, Nxy, Ny, N. $$
Similarly, replacing each instance of $x^2w$ with $x^2zw$ yields a hamiltonian cycle in $\Cay(G/N;S)$
for the third generating set
that passes through the vertices of this graph in the same order as the hamiltonian cycle for the first generating set.
Since $z$ is in the center of $G$ and $z$ appears exactly twice in the list of edges, the endpoints of these cycles in $G$ are the same and they are given by
	$$ x^2wxwyxy=x^2 wxwxz= x^2wx^2 w^k z=x^4 w^{k-1}z =w^{k-1}z. $$
Since $k-1$ is coprime to~$p$, and $z$~has order 2,
this generates $N=\langle wz\rangle$.
So \cref{FGL} provides a hamiltonian cycle in $\Cay(G;S)$.

For the second generating set,
	$$ \bigl( x^2w, x, x^2w, xy, x^3, (xy)^{-1} \bigr)$$
is a hamiltonian cycle in $\Cay(G/N;S)$ that passes through the vertices of this graph in the order:
	$$N, Nx^2, Nx^3, Nx, Nx^2y, Nx^3y, Ny, Nxy, N. $$
Similarly, replacing each instance of $x^2w$ with $x^2zw$ yields a hamiltonian cycle in $\Cay(G/N;S)$ for the fourth
generating set that passes through the vertices in the same order as the hamiltonian cycle for the second generating set.
Since $z$ is in the center of $G$ and $z$ appears exactly twice in the list of edges, the endpoints of each cycle are the same and is
given by
	$$x^2wx^3wxyx^3y^{-1}x^{-1}= x^2wx^3 w xyx^2yz = x^2wx^3wx^3z = x^2wx^2 w^{-k}z= w^{-(k+1)}z. $$
Since $z$ has order 2 and $-(k+1)$ is coprime to $p$ (recall $2\le k\le p-2$), this generates $N=\langle wz\rangle$.
So \cref{FGL} provides a hamiltonian cycle in $\Cay(G;S)$.
\end{proof}

\section{\texorpdfstring{The case where $G/G' \iso \integer_2 \times \integer_2 \times \integer_2$}{The case where G/G' = Z2 x Z2 x Z2}} \label{Z2xZ2xZ2Sect}

\begin{prop}
Assume $|G| = 16p$.
If $G/G' \iso \integer_2 \times \integer_2 \times \integer_2$, then $\Cay(G;S)$ has a hamiltonian cycle.
\end{prop}

\begin{proof} We proceed via case-by-case analysis.

\setcounter{case}{0}

\begin{case}
Assume $\#S = 3$.
\end{case}
Write $S = \{a,b,c\}$. Since $G/G' \iso \integer_2 \times \integer_2 \times \integer_2$, it is easy to see that the sequence
	\begin{align*}
	(a, b, a, c, a, b, a, c)
	\end{align*}
is a hamiltonian cycle in $\Cay(G/G'; S)$. Also, since every nontrivial element of $G/G'$ has order~$2$, we know $s^{-1} \equiv s \pmod{G'}$, for every $s \in S$, so, for any choice of $i_1,\ldots,i_8 \in \{\pm1\}$,
	\begin{align} \label{abacabacCycle}
	\text{$(a^{i_1}, b^{i_2}, a^{i_3}, c^{i_4}, a^{i_5}, b^{i_6}, a^{i_7}, c^{i_8} )$
	is a hamiltonian cycle in $\Cay(G/G'; S)$.}
	\end{align}

Now:
	\begin{itemize}
%	\item If $|a| = p$, then \fullcref{NormalEasy}{p} applies.
	\item If $|a| = 2p$, then $a$~has order~$2$ in $G/P$, but not in~$G$, so \cref{FGL(order2)} and  \cref{QuotientIs4pOr8p} apply.
	\item If $|a| = 4p$, then $a^2$ generates~$G'$. 
	Since $\langle a\rangle/G'$ is normal in $G/G'\iso (\integer_2)^3$, we have $\langle a \rangle\normal G$.
Choose $\beta,\gamma\in\{\pm 1\}$ such that 
		$$ \text{$x^b = x^\beta$ and $x^c = x^\gamma$, for all $x\in\langle a\rangle$.} $$
	Then, letting $i_k = 1$ for $k \notin \{1,3,5\}$, the  endpoint of the path \pref{abacabacCycle} in~$G$ is 
		$$
		\text{$a^{i_1} b a^{i_3} c a^{i_5} b a c
		= a^{i_1} a^{\beta i_3} a^{\beta\gamma i_5} \, g \in G' = \langle a^2 \rangle$,
		\ where $g = bcbac$}
		.$$
	Since each of $i_1,i_3,i_5$ can be $\pm1$ independently, the endpoints that can be obtained in this way are: 
		$$ a^{-3} g, \ a^{-1} g, \ a g, \ a^3 g  .$$
Since  $\langle a^2 \rangle \iso \integer_{2p}$, and at least one of any $4$ consecutive integers is relatively prime to~$2p$, it must be the case that at least one of these endpoints generates $\langle a^2 \rangle = G'$. So \cref{FGL} applies.
	\end{itemize}
Thus, we may assume no element of~$S$ has order divisible by~$p$.
Therefore $s^2 \in Q'$ for every $s \in S$. So we may assume
	$$ \text{every element of~$S$ has order~$2$} $$
for otherwise \cref{FGL(order2)} and  \cref{QuotientIs4pOr8p} apply.

Let 
	$$ g' = a b a c a b a c 
	= [a,b] \, b c \, [a,b] \, b c $$
 be the endpoint of the path \pref{abacabacCycle}  in~$G$.

\begin{obs} For future reference, we note:
	\begin{enumerate}
	\item Since $Q_8$ is not generated by elements of order~$2$, we know $Q$ is not the group described in \pref{ListOrder16-Z2xZ2xZ2-Q8xZ2} of \cref{ListOrder16}.
	\item Suppose $Q$ is the group described in \pref{ListOrder16-Z2xZ2xZ2-Z2x(Z2xZ4)} of \cref{ListOrder16}. 
Let $\overline S$ be the image of~$S$ in~$Q$. It is not difficult to see that $xyz$ is the only element of order~$2$ that is of the form $x^i y^j z$. Thus, we must have $x y z^{\pm1} \in S$, and the other two elements of~$S$ must be in $\langle x, y \rangle$. Since all elements of~$S$ have order~$2$ (and $|xy| = 4$), we conclude that $S$ is of the form
	$$ \overline{S} = \{\, xz^{2i}, yz^{2j}, xyz^{\pm1} \,\} .$$
Up to automorphism (replacing $x$ with $x z^{2i}$, replacing $y$ with $y z^{2j}$, and, if necessary, replacing $z$ with~$z^{-1}$), we have
	\begin{align} \label{S=xz,yz,xyz}
	\text{$\overline{S} = \{\, x, y, xyz \,\}$
	\qquad (if $Q = \integer_2 \ltimes (\integer_2 \times \integer_4)$)}
	. \end{align}
	\end{enumerate}
\end{obs}

\begin{subcase}
Assume $[b,c]$ generates $G'$.
\end{subcase}
Since $b$ and $c$ both have order 2, they generate a dihedral group. Since $G'=\langle (bc)^2\rangle$ has order $2p$, we know
$\langle bc\rangle$ has order $4p$ and $\langle b,c\rangle\iso D_{8p}$.
Thus $b$ and $c$ both invert $\langle bc\rangle$ and $G'$.

Therefore $bc$ centralizes~$G'$, so 
	$g' = [a,b]^2 [b,c]$. 
So \cref{FGL} applies unless $P \subset \langle [a,b] \rangle$ (which implies that $a$ inverts~$P$), and 
	\begin{align} \label{bcbc=baba2}
	 [b,c] \equiv [a,b]^{-2} \pmod{Q'}
	 . \end{align}
Replacing $Q$ by a conjugate, we may assume $b \in Q$. Write $a = \overline a w$ and $c = \overline c w'$ with $\overline a , \overline c \in Q$ and $w,w' \in P$.
Since $a$ and $c$ both invert~$P$, we know $\overline{a}$ and $\overline{c}$ both invert~$P$.

We have
	$$ [b,c] = (bc)^2 = (b \overline c w')^2 =  (b \overline c)^2 (w')^2$$
and
	$$ [a,b] = (ab)^2 = (\overline a w b)^2 = (\overline a b)^2 w^{-2} 
	\text{\ (since $\overline a$ and~$b$ invert~$P$)}, $$
so \pref{bcbc=baba2} tells us that 
	\begin{align} \label{w'=w2}
	w' = w^2
	. \end{align}

\begin{subsubcase}
Assume $[a,b]$ generates $G'$.
\end{subsubcase}
We may interchange $a$ and~$c$, so the preceding calculations tell us that $w = (w')^2 = (w^2)^2 = w^4$, so we must have $p = 3$.

\begin{subsubsubcase}
Assume $Q$ is as described in \fullcref{ListOrder16}{Z2xZ2xZ2-Z2x(Z2xZ4)}.
\end{subsubsubcase}
From \pref{S=xz,yz,xyz} and \pref{w'=w2}, we see that
	$$ \text{$S = \{xw, y, xyzw^{-1}\}$, \quad where $P = \langle w \rangle$} .$$

Let $N=\langle xy \rangle =\{e,xy,z^2,\allowbreak xyz^2\}$ be the cyclic group of order 4 generated by $xy$.
Observe that $N\normal G$ and that $\Cay(G/N;S)$ is graph
isomorphic to $\Cay(D_{12};\{R,F\})$, where $\{R,F\}$ is the natural generating set for $D_{12}$,
under the vertex identification $\phi:G/N\rightarrow D_{12}$ given by 
	\begin{align*}
	\phi(N(xwy)^k ) &=R^{2k}, \\
	\phi(N(xwy)^k (xw) )&=R^{2k+1}, \\
	\phi(N(xwy)^k(xyzw^{-1}) )&=R^{2k}F, 
	\text{ and} \\
	\phi(N(xwy)^k (xw)(xyzw^{-1}) )&=R^{2k+1}F,
	\end{align*} for
any integer~$k$.
The natural hamiltonian cycle $(R^5,F)^2$ in $\Cay(D_{12};\{R,F\})$ corresponds to the hamiltonian
cycle 
	$$((xw,y)^3\#,xyzw^{-1},(y,xw)^3\#,xyzw^{-1})$$
in $\Cay(G/N;S)$.
The endpoint in $G$ is
\begin{align*}
(xwy)^{3}(y^{-1})(xyzw^{-1}) (yxw)^{3}(xw)^{-1}(xyzw^{-1})  &= (xyz^2)(y)(xyzw^{-1})(xy)(w^{-1}x)(xyzw^{-1}) \\
&= (xyz^2)(xz^{-1}w^{-1})(xy)(yzw) \\
&= (yz^{-1}w^{-1})(xzw)  =xyz^2 =(xy)^{-1},
\end{align*}
which generates $\langle xy \rangle =N$.
Thus, \cref{FGL} provides a hamiltonian cycle in
$\Cay(G;S)$.

\begin{subsubsubcase}
Assume $Q = D_8 \times \integer_2 = \langle f, t \rangle \times \integer_2$.
\end{subsubsubcase}
Since $[\overline a,b]$ is nontrivial, we may assume $\overline a = f$ and $b = ft$. 
Because $\langle b,c\rangle$ is a dihedral group and $S$ is minimal, we must have $\overline{c}=ft^{i}z$, for some integer $i$.
Since $[b,\overline{c}]=(b\overline{c})^2=(ftft^{i} z)^2=t^{2i-2}$ is nontrivial, $i$ must be even.
We may replace $z$ with $t^2 z$ since the order of $t^2 z$ is 2 and $t^2 z\in Z(Q)$.
Thus we may assume $\overline{c}=fz$. Then, from \pref{w'=w2}, we see that 
	$$ S = \{f w, ft, f z w^{-1}\} ,$$
 where $f$ inverts $w$, whereas $t$ and~$z$ centralize~$w$ (and $w$ is a generator of~$P$).

Let $N=\langle tz \rangle =\{e,tz,t^2,t^3 z\}$ be the cyclic group of order 4 generated by $tz$.
Observe that $N\normal G$ and that $\Cay(G/N;S)$ is graph
isomorphic to $\Cay(D_{12};\{R,F\})$, where $\{R,F\}$ is the natural generating set for $D_{12}$,
under the vertex identification $\phi \colon G/N\to D_{12}$ given by 
	\begin{align*}
	\phi(N((ft)(fzw^{-1}))^k )&=R^{2k}, \\
	\phi(N((ft)(fzw^{-1}))^k (ft) )&=R^{2k+1},\\
	\phi(N((ft)(fzw^{-1}))^k(fw) )&=R^{2k}F,
	\text{ and} \\
	\phi(N((ft) (fzw^{-1}))^k (ft)(fw) )&=R^{2k+1}F,
	\end{align*}
for any integer~$k$.
The natural hamiltonian cycle $(R^5,F)^2$ in $\Cay(D_{12}; \{R,F\})$ corresponds to the hamiltonian
cycle 
	$$((ft,fzw^{-1})^3\#,fw,(fzw^{-1},ft)^3\#, fw)$$
in $\Cay(G/N;S)$.
The endpoint in $G$ is
\begin{align*}
(ftfzw^{-1})^{3}(fzw^{-1})^{-1}(fw) (fzw^{-1}ft)^{3}(ft)^{-1}(fw)  &= (tz)(wz^{-1}f^{-1})(fw) (t^3 z)(t^{-1}f^{-1})(fw) \\
&= (tz)(wz^{-1}w)(t^{3}z)(t^{-1}w) \\
&= (tw^2)(t^2 zw)  =t^3 z =(tz)^{-1},
\end{align*}
which generates $\langle tz \rangle =N$.
Thus, \cref{FGL} provides a hamiltonian cycle in
$\Cay(G;S)$.

\begin{subsubcase}
Assume $[a,b]$ does not generate $G'$.
\end{subsubcase}
Because we could interchange $b$ and~$c$, we may assume $[a,c]$ also does not generate $G'$. Since 
	$$[a,c] = (ac)^2 = (\overline a w \, \overline c w^2)^2 = (\overline a \, \overline c w)^2 = (\overline a \, \overline c)^2 w^2 ,$$
and $w^2$ generates~$P$, this implies that $\overline a$ commutes with~$\overline c$. By the same argument, $\overline a$ commutes with~$b$. So $\overline a$ is in the center of~$Q$. Therefore $Q = \langle b, \overline c \rangle \times \langle \overline a \rangle$. Looking at the list of groups in \fullcref{ListOrder16}{Z2xZ2xZ2} (and recalling that $\overline a$, $b$, and~$\overline c$ all have order~$2$), we conclude that $Q = D_8 \times \integer_2$. Furthermore, we have $\overline a = z$, and $\langle b, \overline c \rangle = D_8$, so we may assume
	$$S = \{ z w, f, ftw^2 \} $$
where $f$ and~$z$ invert~$w$, and $t$~centralizes~$w$ (and $w$ is the generator of~$P$).

Let $N=Q'=\langle t^2\rangle=\{e,t^2\}$. Observe that the graph $\Cay(G/N;S)$ is isomorphic to $\Cay(\integer_{2}\ltimes\integer_{4p};\{X,Y\})$, where $\{X,Y\}$ is the natural generating set for
$\integer_{2}\ltimes\integer_{4p}$ given by 
	$$\integer_{2}\ltimes\integer_{4p}=\langle X,Y| X^{2}=Y^{4p}=1, X^{-1}YX=Y^{2p-1}\rangle .$$
This graph isomorphism is given by the vertex identification 
$\phi \colon G/N\to \integer_{2}\ltimes\integer_{4p}$ where
	\begin{align*}
	\phi(N(fftw^2)^k ) & =Y^{2k} , \\
	\phi(N(fftw^2)^k f )&=Y^{2k+1} , \\
	\phi(N(fftw^2)^k(zw) )&=Y^{2k}X,
	\text{ and} \\
	\phi(N(fftw^2)^k(f)(zw) ) &=Y^{2k+1}X, 
	\end{align*}
for any integer $k$.
The hamiltonian cycle $(Y^{2p-1},X)^4$ in $\Cay(\integer_{2}\ltimes\integer_{4p};\{X,Y\})$ corresponds to the hamiltonian
cycle 
	$$((f,ftw^{2})^p\#,zw,(ftw^{2},f)^p\#, zw)^2$$
in $\Cay(G/N;S)$.
The endpoint in $G$ is
\begin{align*}
\bigl( (fftw^{2})^{p}(ftw^{2})^{-1}(zw) (ftw^{2}f)^{p}(f)^{-1}(zw) \bigr)^2  &= 
\bigl( (t^p)(w^{-2}t^{-1}f^{-1})(zw) (t^{-p})(f^{-1})(zw) \bigr)^2 \\
&= \bigl( (ft^{-p+1}w^{2})(zw)(ft^{p})(zw) \bigr)^2 \\
&= (t^{2p-1})^2  =t^2,
\end{align*}
which generates $\langle t^2 \rangle =N$.
Thus, \cref{FGL} provides a hamiltonian cycle in
$\Cay(G;S)$.

% %%% Here's a simpler proof, but it doesn't work for p = 7
%The cycle $\bigl( (f, ftw^2)^4 \#, zw\bigr)^2$ is a hamiltonian cycle in $\Cay(G/P)$, and its endpoint in $G$ is
%\begin{eqnarray*}
%\bigl((fftw^2)^4(ftw^2)^{-1}zw\bigr)^2&=& (w^8w^{-2}t^{-1}fzw)^2\\
%&=& (w^7ftz)^2\\
%&=& w^{14}.
%\end{eqnarray*}
%If $p\neq 7$, then $w^{14}$ generates $P$, so Lemma 1.11 provides a hamiltonian cycle in $\Cay(G;S)$.  

\begin{subcase}
Assume there do not exist $s,t \in S$, such that $\langle [s,t] \rangle = G'$.
\end{subcase}
Then, since 
	$$ \langle [a,b], [a,c], [b,c] \rangle = G' = Q' \times P ,$$
we may assume $\langle [a,b] \rangle = P$ and $\langle [b,c] \rangle = Q'$. 

Furthermore, we may assume $bc$ inverts~$G'$, for otherwise it centralizes~$G'$, in which case, $g' = [a,b]^2 [b,c]$, so $g'$ generates $G'$, so \cref{FGL} applies.

\begin{subsubcase}
Assume $Q = D_8 \times \integer_2$.
\end{subsubcase}
Since $\langle b,c\rangle$ is dihedral and $Q'=\langle [b,c]\rangle=\langle (bc)^2\rangle$ has order 2, it must be the case that $ab$ has order~$4$ and
$\langle b,c\rangle\iso D_{8}$. So we may assume $b=f$ and $c=ft$.

Write $a = \overline a w$ with $\overline a \in Q$ and $w \in P$. Since $a$~and~$b$ have order~$2$, we know they generate a dihedral group, so $\overline a$ and~$b$ both invert~$P$, but $c$~centralizes~$P$ (since $bc$ inverts~$G'$).

Since $[a,b] \in P$ projects trivially into~$Q$, we know $\overline a$ commutes with~$b = f$, so $\overline a \in \langle f, t^2 \rangle z$. However, since $t^2 z$ is in the center of~$Q$, there is no harm in replacing $z$ with~$t^2 z$, so we may assume $\overline a \in \langle f \rangle z$. Thus, there are only two generating sets to consider:
$$ \text{$S = \{zw, f, ft \}$ or $S = \{fzw, f, ft\}$ .}$$
In each case, assume the first two generators invert~$P$, and the third generator centralizes~$P$. (Thus, in both cases, $f$~and~$t$~invert~$P$. However, $z$~inverts~$P$ in the first case, but $z$~centralizes~$P$ in the second case.)

The cycle $\bigl( (ft, f)^4 \#, a\bigr)^2$ in each of these generating sets (so in the first, $a=zw$, and in the second, $a=fzw$), is a hamiltonian cycle in $\Cay(G/P;S)$, and its endpoint in $G$ is $\bigl((ftf)^4f^{-1}a\bigr)^2=(fa)^2$.   In the first case, the endpoint $(fa)^2$ is $(fzw)^2=f^2z^2w^2=w^2$, while in the second case, it is $(ffzw)^2=(zw)^2=w^2$.  Since $w^2$ certainly generates~$P$, \cref{FGL} provides a hamiltonian cycle in $\Cay(G;S)$.

\begin{subsubcase}
Assume $Q \neq D_8 \times \integer_2$.
\end{subsubcase}
We will show that this case cannot occur.
From \pref{S=xz,yz,xyz}, it is easy to see that no two elements of~$\overline{S}$ commute. So the image of $[a,b]$ in~$Q$ is nontrivial, which contradicts the fact that $\langle [a,b] \rangle = P$.

\begin{case}
Assume $\#S >3$.
\end{case}
Since $G/ P \iso Q$ is a $3$-generated $2$-group, there is a $3$-element subset $\{a,b,c\}$ of~$S$ that generates $G/P$. The minimality of~$S$ implies $\langle a,b,c \rangle \neq G$, so $|\langle a,b,c \rangle| = 16$. Thus, we may assume $\langle a,b,c \rangle = Q$.

Since $a$, $b$, and~$c$ all have order~$2$ in $Q/Q'$, they also have order~$2$ in $G/Q'$. So we may assume
	$$ \text{$a$, $b$, and~$c$ all have order~$2$} $$
for otherwise \cref{FGL(order2)}  and  \cref{QuotientIs4pOr8p} apply.

Let $s$ be the 4th element of~$S$. We may assume $s \notin G'$, for otherwise \cref{CxLHamCyc} and  \cref{QuotientIs4pOr8p} apply with $X = \{s\}$.

Let $s=wq$ where $w$ generates $P$ and $q\in Q$. If $c$ centralizes $P$, we have
$$
[c,s]=c^{-1}(wq)^{-1}cq = (q^{-1})^c (w^{-1})^c wq= (q^{-1})^c w^{-1}wq =(q^{-1})^c q=[c,q]\in Q'.
$$
Since 
	$$G'=Q'\times P = \langle [a,b], [a,c], [b,c], [a,s], [b,s], [c,s]\rangle$$
 and $Q'=\langle [a,b], [a,c], [b,c]\rangle$, we may assume $P\subset \langle [c,s]\rangle$.
Thus $[c,s]\not\in Q'$ and $c$ inverts $P$. Therefore, $P\subset \langle [c,s]\rangle \subset \langle c,s\rangle$.

We claim that $s \in c G'$. If not, then the image of $\langle c,s\rangle$ in $G/G'$ has order~$4$, so we may assume $\{ a,c,s\}$ generates $G/G'\iso Q/Q'=Q/\Phi(Q)$.
But this implies that $\{ a,c,s\}$ generates $G/P\iso Q$. Since we also know that $P \subset \langle c,s \rangle \subset \langle a, c,s \rangle$, we conclude that $\langle a, c, s \rangle = G$. This contradicts the minimality of~$S$.

We may assume $s \notin cP$, for otherwise \cref{FGL(gen-twice)}  and \cref{QuotientIs4pOr8p} apply. 
Let $u$ be a generator of $Q'\iso\integer_2$. Then $s\in c\, G' = c\, Q'P =\{c,cu\}P$.
Since $s\not\in c\, P$ and $s\not\in Q$, we have $s=cuw$ for some generator~$w$ of $P\iso\integer_p$.
Let $\gamma=uw$. Then $s=c\gamma$ and
	$$\langle  \gamma\rangle  =\langle uw\rangle = \langle u\rangle \langle w\rangle= Q'P =  G' .$$
Since $[c,s]$ generates~$P$, we see that $c$ inverts~$P$. 

We claim that both $a$ and $b$ centralize $P$. If not, we may assume $a$ inverts $P$. Then $P\subset \langle [a,s]\rangle\subset \langle a,s\rangle$.
Since $\{ a,b,s\}$ and $\{ a,b,c\}$ have the same image in $G/G'\iso Q/Q'=Q/\Phi(Q)$, we know $\{ a,b,s\}$ generates $G/G'$.
This implies that $\{ a,b,s\}$ generates $G/P\iso Q$. Hence, $\{a,b,s\}$ generates $G$, contradicting the minimality of $S$.

Now, since $G/G' \iso \integer_2 \times \integer_2 \times \integer_2$ (and $s \equiv c \pmod{G'}$), it is easy to see that all three of the following sequences are hamiltonian cycles in $\Cay(G/G';S)$:
	$$ (a,c,b,c,a,c,b,c), \  (a,c,b,c,a,c,b,s), \  (a,c,b,s,a,c,b,s) .$$
Let $g = acbcacbc$ be the endpoint (in~$G$) if the first cycle. Then the endpoints of the other two cycles are:
	$$ \text{$g c^{-1} s = g \, \gamma$ \ and \ $acbsacbs = acb(c\gamma)acb(c\gamma) = g\, \gamma^2$} .$$
Now $g \in Q' = \langle \gamma^p \rangle$, and $|\gamma| = 2p$. Now, it is easy to see that if $m$ is a multiple of~$p$, then either $m+1$ or $m+2$ is relatively prime to $2p$. Therefore, either $g\, \gamma$ or $g \, \gamma^2$ generates $\langle \gamma \rangle = G'$, so \cref{FGL} applies.
\end{proof}

\end{document}